\documentclass[10pt]{article}
\usepackage{color}
\usepackage{amssymb}
\usepackage{amsthm,array,amssymb,amscd,amsfonts,latexsym, url}
\usepackage{amsmath}
\usepackage[all]{xy}
\newtheorem{theo}{Theorem}[section]
\newtheorem{prop}[theo]{Proposition}

\newtheorem{lemm}[theo]{Lemma}
\newtheorem{coro}[theo]{Corollary}
\newtheorem{rema}[theo]{Remark}
\newtheorem{Defi}[theo]{Definition}

\newtheorem{conj}[theo]{Conjecture}

\newtheorem{question}[theo]{Question}
\newtheorem{example}[theo]{Example}

\title{Triangle varieties and surface decomposition of    hyper-K\"ahler manifolds}
\author{Claire Voisin}

\date{}

\newfont{\gothic}{eufb10}
\begin{document}
\maketitle
\setcounter{section}{-1}

\begin{abstract}  We introduce and study  the notion of
``surface decomposable'' variety, and discuss the possibility that any projective hyper-K\"ahler manifold is surface decomposable, which would produce new evidence for Beauville's weak splitting conjecture. We show that surface decomposability relates to
the Beauville-Fujiki relation, a constraint on the cohomology ring of the variety,  and that   general varieties with $h^{2,0}\not=0$
are not surface decomposable.
We also formalize the notion of triangle variety that is  useful to produce surface decomposition.
We show the existence of these  geometric structures on most explicitly constructed  classes of projective hyper-K\"ahler manifolds of Picard number $1$.

\end{abstract}
\section{Introduction}
Let $X$ be a complex manifold of dimension $2n$ equipped with a
holomorphic $2$-form $\sigma_X$ which is everywhere of maximal rank $2n$.
Locally for the Euclidean  topology on $X$, Darboux theorem tells that
one can write, for an adequate choice of holomorphic coordinates,
$$\sigma_X=\sum_{i=1}^{n} dz_i\wedge dz_{n+i},$$
that is, $\sigma_X$ is the sum of $n$ closed holomorphic $2$-forms of rank $2$.

A natural question is whether this statement can be made more global, particularly in the projective case:

\begin{question} \label{question01} Does there exist a generically finite cover
$\phi: Y\rightarrow X$, such that $\phi^*\sigma_X$ is the sum of
$n$ closed holomorphic $2$-forms of rank $2$  on $Y$?
\end{question}

Our goal in this paper is to study a geometric variant of this question,
 namely the possibility that any projective hyper-K\"ahler manifold is {\it surface decomposable}
 (or admits a {\it surface decomposition})
in the following sense:
\begin{Defi}\label{defiintrodecomp} A  smooth projective variety $X$ of dimension $2n$ will
 be said surface decomposable if
there exist a smooth variety $\Gamma$, smooth projective surfaces $S_1,\ldots,\,S_n$, and
generically finite surjective morphisms
$\phi:\Gamma\rightarrow X,\,\psi:\Gamma\rightarrow S_1\times\ldots\times\,S_n$ such that for any
holomorphic $2$-form $\sigma\in H^0(X,\Omega_X^2)$,
\begin{eqnarray}\label{eqpourdecintro} \phi^*\sigma=\psi^*(\sum_i{\rm pr}_i^*\sigma_i)
\end{eqnarray}
for some holomorphic $2$-forms $\sigma_i$ on $S_i$.

\end{Defi}
We will show that surface decomposability is restrictive for general projective varieties of dimension
$2n\geq 4$ (see Theorem \ref{coronotdecomp}).
The first examples of hyper-K\"ahler manifolds were constructed by Beauville \cite{beaujdg} and Fujiki
\cite{fujiki} as punctual Hilbert schemes of  $K3$ surfaces
or abelian surfaces and hence were rationally dominated by products of surfaces. They were thus obviously
surface decomposable.  However, it follows from deformation theory that these $K3$ or abelian
surfaces disappear under a general deformation to a projective hyper-K\"ahler manifold with Picard number $1$. Indeed, the parameter space
for $K3$ surface is too small to parameterize also general deformations with Picard number $1$  of their punctual Hilbert schemes. The starting point of this paper is
the observation that on many explicitly described general deformations as above, a surface decomposition still exists.

Let us make several  remarks concerning Definition \ref{defiintrodecomp}. First of all,
the condition (\ref{eqpourdecintro}) has been asked only for holomorphic $2$-forms, but
 by an elementary argument
involving Hodge structures (see Section \ref{secdecomp}), it then follows that it is satisfied for any
transcendental class $\eta\in H^2(X,\mathbb{Q})_{tr}$, the later space
 being defined as the smallest Hodge substructure of $H^2(X,\mathbb{Q})$ whose
 complexification contains $H^{2,0}(X)$.

Next, if
we allow an arbitrarily large number $N$ of summands $S_i$ and
 only  ask that $\phi$ is surjective and $\psi$ is generically finite, then
the definition is (at least conjecturally) not restrictive since  the property is satisfied by any smooth projective variety $X$ satisfying the Lefschetz standard conjecture in degree $2$ (see Proposition \ref{prolefschetz}).
Similarly, we could consider, instead of (\ref{eqpourdecintro}), the weaker  condition
that $\phi$ is surjective  generically finite and
\begin{eqnarray}\label{eqweakdecomp} \sigma=\phi_*(\psi^*(\sum_i{\rm pr}_i^*\sigma_i))\,\,{\rm in}\,\,H^{2,0}(X),
\end{eqnarray}
but, as before, this is implied by Lefschetz standard conjecture
 (and we can then  take $n=1$ to achieve (\ref{eqweakdecomp})).

The
reason why (\ref{eqweakdecomp}) is much weaker than (\ref{eqpourdecintro}) is the fact that
 pull-back maps
are compatible with cup-products, while push-forward maps are not. More precisely, we will
show  (see Proposition \ref{propquadricfujiki}) that  the surface decomposability implies and provides  a geometric explanation for
Beauville-Fujiki's famous formula for the self-intersection of degree $2$ cohomology on a
hyper-K\"ahler manifold:
$$\int_X\alpha^{2n}=\lambda q(\alpha)^n,$$
where $q$ is  the Beauville-Bogomolov quadratic form on $H^2(X,\mathbb{Q})_{tr}$ (see \cite{beaujdg}). To prove this implication, we have to assume that
the Mumford-Tate group of the Hodge structure on  $H^2(X,\mathbb{Q})_{tr}$ is large enough to guarantee that
 all the quadratic forms on $H^2(X,\mathbb{Q})_{tr}$
 induced from $(\,,\,)_{S_i}$ via the morphism of Hodge structures $\psi_i: \eta\mapsto \eta_i$
 (see Section \ref{secdecomp})
 are proportional, but this is automatic if $X$ is the general member of a family of polarized hyper-K\"ahler manifolds.

 This observation suggests that surface decomposability could be a way to approach
 the weak splitting property of hyper-K\"ahler manifolds (see Conjecture \ref{conjws})
 conjectured by Beauville in \cite{beau}. It says that cohomological
  polynomial relations between divisors on
 hyper-K\"ahler manifolds $X$ are satisfied on the Chow level. A weaker version asks that
 $X$ has a canonical $0$-cycle $o_X\in{\rm CH}_0(X)$ such that
 $D^{2n}$ is proportional to $o_X$ in ${\rm CH}_0(X)$ for any divisor $D$ on $X$.

 In this direction, we prove  Theorem \ref{prodecompws} which has the following
 consequence:
  \begin{theo} \label{theoconsdeprodecompws} Assume that
 the general member $X_t$  of a  family $(X_t)_{t\in B}$ of hyper-K\"ahler manifolds
 with given Picard lattice $\Lambda$ has  a surface decomposition  with a simply connected $\Gamma$.
 Then  the
 Beauville  weak splitting property holds   for  the divisor  classes belonging to
 ${\rm NS}(X_t)^{\perp \Lambda}$  if and only if there exists a $0$-cycle
 $o_{X_t}\in {\rm CH}_0(X_t)$ such that
 $D^{2n}$ is proportional to $o_{X_t}$ in ${\rm CH}_0(X_t)$ for any divisor
 $D\in {\rm NS}(X_t)^{\perp \Lambda}$.
 \end{theo}
The importance of   Theorem \ref{theoconsdeprodecompws}  is that it reduces the weak
  splitting property to
 checking the weak version, namely in top degree.
 The  weakness of the result is that it  applies only to divisors of class perpendicular to
 $\Lambda$, which means, in practice, primitive. One has  to understand separately  what happens with
 the powers $h^k$ of the polarizing class. In all the geometric examples we have, the natural  surface decomposition
 that we exhibit
 provides (\ref{eqpourdecintro}) only on primitive cohomology.
\begin{rema}{\rm The statement above is empty for the} very general {\rm member of the family since it  has
${\rm NS}(X_t)^{\perp \Lambda}=0$. The statement above  is interesting for
 special hyper-K\"ahler manifolds with higher Picard rank $\rho\geq \rho_{gen}+2$, which are parameterized
 by
 a countable union of closed algebraic subsets in the base $B$, which is   dense
 in $B$ if ${\rm dim}\,B\geq 2$.}
\end{rema}

The second geometric notion that will play an important role in this paper is the following.

\begin{Defi} \cite{voicoiso} An algebraically coisotropic subvariety of
 a hyper-K\"ahler manifold $X$ of dimension
$2n$ is a subvariety $Z\subset X$ of codimension $k\leq n$ which admits a
rational fibration $\phi:Z\dashrightarrow W$, where $W$ is smooth and ${\rm dim}\,W=2n-2k$, such that
\begin{eqnarray}\label{defalgcoiso} \sigma_{X|Z_{\rm reg}}=\phi_{\rm reg}^*\sigma_{W},
\end{eqnarray}
where $\phi_{\rm reg}:Z_{\rm reg} \dashrightarrow W$ is the restriction of $\phi$ to the regular locus of $Z$ and $\sigma_W$ is  a holomorphic $2$-form on $W$.
\end{Defi}
This notion is to be distinguished from the notion of coisotropic subvariety, which just asks that
the restriction $\sigma_{X|Z_{\rm reg}}$ has rank $2n-2k$ at any point, or equivalently that
\begin{eqnarray}
\label{eqpourkerisot} T_{Z_{\rm reg},x}^{\perp\sigma_X}\subset   T_{Z_{\rm reg},x}
\end{eqnarray}
at any point $x$ of $Z_{\rm reg}$. Equation (\ref{eqpourkerisot}) defines then a foliation on $Z_{\rm reg}$ and $Z$ is
algebraically coisotropic when the leaves of this  foliation are algebraic.
 The two
 notions coincide in the case $n=k$ of Lagrangian varieties. For $k=1$, any divisor
  is coisotropic but smooth ample divisors are not algebraically isotropic
(see \cite{AmCam}).  It is not easy a priori to construct algebraically coisotropic divisors in a projective hyper-K\"ahler manifolds. Examples
are given by uniruled divisors but they are rigid. One open question is whether  a projective hyper-K\"ahler manifold can always be swept out
by algebraically coisotropic divisor. This is certainly true if $X$ has a surface decomposition
(see Proposition \ref{propeasysenseintro} below). In the other direction, we show   in Proposition
\ref{theodecompdiv} that the existence of  a  $1$-parameter family of algebraically coisotropic divisors for $X$
implies a decomposition, on a generically finite cover of $X$, of the holomorphic  $2$-form of $X$ as a sum
of one rank $2$ and one  rank $2n-2$ holomorphic forms.

The theory of  coisotropic subvarieties of higher codimension is  more complicated.
In the paper \cite{voicoiso}, we discussed the constraints on the cohomology classes of coisotropic subvarieties of higher codimension
and asked whether the space of coisotropic classes, namely those satisfying these constraints, are generated by classes of algebraically coisotropic subvarieties.
We also proposed the construction of algebraically coisotropic subvarieties as total spaces of  $2n-2k$-dimensional
families of  constant cycles varieties (in the sense of Huybrechts \cite{huybrechts}) of dimension $k$.

The  third   notion that will be introduced and studied in  this paper is that
   of {\it  triangle variety}:
\begin{Defi} \label{defitriangle} A triangle  variety for  $X$ (equipped with a holomorphic $2$-form
 $\sigma_X$) is a  subvariety of  $X\times X\times X$ which dominates $X$ by the three projections, maps in a generically finite way to its image in
$X\times X$ via the three projections and  is Lagrangian for the holomorphic form
$\sigma_1+\sigma_2+\sigma_3$ on $X^3$, where $\sigma_i={\rm pr}_i^*\sigma$.
\end{Defi}
The following example will be generalized in Section
\ref{secjacfib}.
\begin{example}\label{exintro} {\rm Let $S\rightarrow B$ be an elliptic surface with a section. Then the graph of minus  the relative  sum map
$S\times_BS\dashrightarrow S$, which is naturally  contained in $S^3$, is  a triangle variety for any holomorphic $2$-form
$\sigma$ on $S$.
}
\end{example}
Triangle varieties  seem to exist for  most explicitly constructed classes of projective  hyper-K\"ahler manifolds.
In fact, the simplest example
of them, namely actual triangles in the Fano variety $F_1(Y)$ of lines of a  smooth cubic fourfold $Y$, is studied with detail in
\cite{shenvial} by Shen and Vial,  who use them to study a decomposition (Beauville splitting) of the Chow groups  of $F_1(Y)$.
 The main geometric  examples, including this one,  will be presented
in Section \ref{secexamples}.

The first  link between triangle varieties, surface decompositions and algebraically coisotropic  subvarieties is the following
obvious implication:
\begin{prop}\label{propeasysenseintro} If $X$ has a surface decomposition, then it has mobile  algebraically coisotropic
subvarieties of any codimension $k\leq n$. If the surfaces appearing in a surface decomposition
of $X$ have  triangle varieties, (for example, if they are elliptic,)  then so does $X$.
\end{prop}
\begin{proof} Indeed let $\phi:\Gamma\rightarrow X,\,\psi:\Gamma \rightarrow S_1\times\ldots\times S_n$ be surjective  generically finite  maps such that
$$\phi^*\sigma_X=\psi^*({\rm pr}_1^*\sigma_{S_1}+\ldots+{\rm pr}_n^*\sigma_{S_n})\,\,{\rm in}\,\,H^{2,0}(\Gamma).$$
For any integer $k\leq n$, let $C_i\subset S_i,\,i=1,\ldots,\,k$, be very ample curves in general position. Then
$\phi(\psi^{-1}(C_1\times\ldots C_k\times S_{k+1}\times\ldots\times S_n))$ is an algebraically coisotropic
subvariety of $X$ of codimension $k$. It $T_i\subset  S_i^3$ are triangle varieties,
then
$\phi^3(T_1\times\ldots\times T_n)$ is a triangle variety for $X$.
\end{proof}
 We will show in the paper   how conversely triangle subvarieties for $X$ and algebraically coisotropic
 subvarieties of $X$  of codimension $n-1$, where ${\rm dim}\,=2n$,  can be used
to construct surface decompositions and algebraically coisotropic varieties  of  $X$ of any codimension $1\leq k\leq n$.
We prove the following result (see Theorem \ref{proexistesurfdec}).
\begin{theo}\label{proexistesurfdecintro} Let $X$ be a  projective hyper-K\"ahler variety of dimension $2n$. Assume $X$ has a triangle variety
$T\subset X^3$ and  algebraically coisotropic subvarieties $\tau:Z\dashrightarrow \Sigma$ of dimension $n+1$. Let $F\subset X$ be the general fiber of $\tau$.
 Then if the intersection number
$F^{n}\cdot p_{1\ldots n}(T_{n+1})$ of cycles in $X^n$ is nonzero,
 $X$ admits a surface decomposition.
\end{theo}
  Finally we will
prove (see Theorem \ref{theodec}), as an application of  Theorem
\ref{proexistesurfdecintro} or  variants of it, that most hyper-K\"ahler manifolds
that have been explicitly
constructed from algebraic geometry admit surface decompositions.
\section{The decomposition problem for hyper-K\"ahler varieties\label{secdecomp}}
Let $X$ be a smooth projective manifold  of dimension $2n$ (we will  later on  focus  on  the hyper-K\"ahler case). We wish to study the existence and consequences of a surface decomposition
of the form described in the introduction (see Definition \ref{defiintrodecomp}),
namely the existence of  smooth projective surfaces $S_i,\,i=1,\ldots , \,n$, and
an effective correspondence $\Gamma$ (which can be assumed to be smooth and projective)
$$\phi: \Gamma {\rightarrow} X,\,\,\,\,\,\psi:
\Gamma {\rightarrow} S_1\times\ldots\times S_n,$$
with $\phi$ and $\psi$ dominant  generically finite, such that for any
$\sigma_X\in H^{2,0}(X)$,
\begin{eqnarray} \label{eqdecompholo} \phi^*\sigma_X=\psi^*(\sum_{i=1}^n{\rm pr}_i^*\sigma_{S_i})\,\,
{\rm in}\,\,H^{2,0}(\Gamma),
\end{eqnarray}
for some $\sigma_{S_i}\in H^{2,0}(S_i)$, where the
 ${\rm pr}_i:S_1\times\ldots\times  S_n\rightarrow S_i$ are  the various projections.
 This implies in fact that we have more generally, for any
  $\eta\in H^2(X,\mathbb{Q})_{tr}$,  a relation
 \begin{eqnarray} \label{eqdecomp} \phi^*\eta=\psi^*({\rm pr}_1^*\eta_1+\ldots + {\rm pr}_n^*\eta_n)\end{eqnarray}
 for some $\eta_i\in H^2(S_i,\mathbb{Q})_{tr}$.
 Indeed, the form  $\sigma_{S_i}\in H^{2,0}(S_i)$ in (\ref{eqdecompholo}) can be reconstructed
from $\sigma_X$ by the action of the  morphism
 of Hodge structures
$$\psi_i: H^2(X,\mathbb{Q})_{tr}\rightarrow H^2(S_i,\mathbb{Q}),\,\,\,\eta\mapsto \eta_i$$
\begin{eqnarray}\label{eqpourpsii}\psi_i(\eta)=\frac{1}{NN_i}{\rm pr}_{i*}( d^{2n-2}\cup \psi_*(\phi^*\eta))  ,
\end{eqnarray}
where $N$ is the degree of $\psi$ and $d=\sum_i{\rm pr}_i^*d_i$ is the first Chern class of an ample divisor on $\prod_iS_i$ with the property that
${\rm pr}_{i*}(d^{2n-2})=N_i 1_{S_i}$ in $H^0(S_i,\mathbb{Q})$ for all $i$.
 The last condition indeed guarantees
that
$$\frac{1}{NN_i}{\rm pr}_{i*}(d^{2n-2}\cdot \psi_*\psi^*(\sum_j{\rm pr}_j^*\eta_j))=\eta_i$$
for any cohomology classes   $\eta_j$ on $S_j$ such that $\eta_i\cup d_i=0$.
The morphisms of Hodge structures
$\psi_i$ being defined as in (\ref{eqpourpsii}), condition (\ref{eqdecompholo}) then rewrites as
$$\phi^*\sigma_X=\psi^*(\sum_i{\rm pr}_i^*(\psi_i(\sigma_X)))\,\,{\rm in}\,\,H^{2,0}(\Gamma).$$
This equality  defines a Hodge substructure  of $H^2(X,\mathbb{Q})$. Hence if it is satisfied
 on $H^{2,0}(X)$, it is satisfied on
$H^2(X,\mathbb{Q})_{tr}$.

 A variant of this definition assumes that $S_1=\ldots=S_n$ and the correspondence
 $\Gamma$ is symmetric
 with respect
  to the symmetric group action on $S^n$, but this is not essential.
  Another variant  asks that condition (\ref{eqdecomp}) is satisfied
  for any $\eta\in H^2(X,\mathbb{Q})_{prim}$,
where the subscript ``prim'' refers to the choice of an ample line bundle $L$ on $X$, and
 primitive cohomology is primitive with respect to
$l=c_1(L)$. If we work with very general hyper-K\"ahler manifolds of Picard number $1$,
the two notions coincide.
In the hyper-K\"ahler case, equation (\ref{eqdecompholo}) decomposes the smooth projective manifold
 $X$ in the sense that the rank $2n$ holomorphic
$2$-form on $X$ gets decomposed as the sum of $n$ (generically) rank $2$ holomorphic $2$-forms
$\psi^*({\rm pr}_i^*\sigma_{S_i})$ on the generically finite cover $\Gamma$.

If we now relax the conditions on $\phi,\,\psi$  in Definition \ref{defiintrodecomp}  and just ask that
$\phi$ is surjective and
$ \phi^*\sigma_X=\psi^*(\sum_{i=1}^N{\rm pr}_i^*\sigma_{S_i})$ for any holomorphic $2$-form on $X$, allowing an arbitrarily
 large number of summands, then
a decomposition as in (\ref{eqdecompholo}) should always exist, for any smooth projective variety $X$. More precisely:
\begin{prop}\label{prolefschetz} Let $X$ be a smooth projective variety.
Assume $X$ satisfies Lefschetz
standard  conjecture for degree $2$ cohomology. Then there is a generically finite cover
$\phi:\Gamma\rightarrow X$, surfaces $S_1,\dots,\,S_N$,  and a morphism
$\psi:\Gamma\rightarrow S_1\times\ldots \times S_N$, such that any $(2,0)$-form $\sigma$ on $X$ satisfies
\begin{eqnarray}\label{eqdeprop}\phi^*\sigma=\psi^*(\sum_i{\rm pr}_i^*\sigma_i)\,\,{\rm in}\,\,H^{2,0}(\Gamma)
\end{eqnarray}
for some $(2,0)$-forms $\sigma_i$ on $S_i$.
\end{prop}
\begin{rema}{\rm The proof will even show that we can take $S_1=\ldots=S_N$ and
$\Gamma$ symmetric.}
\end{rema}
\begin{proof}[Proof of Proposition \ref{prolefschetz}] The Lefschetz standard
conjecture for degree $2$ cohomology on $X$ provides a codimension
$2$-cycle $Z$  on $X\times X$  such that, if  $n={\rm dim}\,X$,
$Z^*:H^{2n-2}(X,\mathbb{Q})\rightarrow H^2(X,\mathbb{Q})$
is the inverse of  the Lefschetz isomorphism  $l^{n-2}:H^2(X,\mathbb{Q})\cong H^{2n-2}(X,\mathbb{Q})$
induced by the first Chern class $l$  of
a very ample line bundle $L$ on $X$.
Let $j:S\rightarrow X$ be a smooth  surface  which is the complete intersection of $n-2$
ample  hypersurfaces in $|L|$. Then by the Lefschetz theorem on hyperplane section, the Gysin morphism
$ j_*: H^2(S,\mathbb{Q})\rightarrow  H^{2n-2}(X,\mathbb{Q})$ is surjective,
so that,
denoting by  $Z_S$ the restriction of $Z$ to
$X\times S$, we find that
$$Z_S^*=Z^*\circ j_*: H^2(S,\mathbb{Q})\rightarrow H^2(X,\mathbb{Q})$$
is also surjective.
We can make  $Z_S$ is effective by replacing if necessary its negative components
 $-Z_{S,i}$ by effective  residual
cycles $Z'_i$ of class $H^{2}-Z_{S,i}$, where $H={\rm pr}_1^*H_1+{\rm pr}_2^*H_2$ is a sufficiently ample line
 bundle on $X\times S$. The cycle $H^2$ acts trivially on transcendental
  cohomology, so this change does not affect
 $Z_S^*:H^2(S,\mathbb{Q})_{tr}\rightarrow H^2(X,\mathbb{Q})_{tr}$.
Because $Z_S$ is effective, it is given by a rational map
$$\phi_{Z_S}:X\dashrightarrow S^{(N)},$$
so that
$$Z_S^*\sigma_S=\phi_{Z_S}^*\sigma_{S^{(N)}} \,\,{\rm in}\,\,H^{2,0}(X),$$
for any holomorphic $2$-form $\sigma_S$ on $S$,
where $\sigma_{S^{(N)}}$ denotes the induced $2$-form on
$S^{(N)}$.  Recall that, denoting by
$\mu:S^N\rightarrow S^{(N)}$ the quotient map,
 \begin{eqnarray}\label{eqnpourformetireeprod} \mu^*\sigma_{S^{(N)}}=\sum_{i=1}^N{\rm pr}_i^*\sigma_S \,\,{\rm in}\,\,H^{2,0}(S^N).
 \end{eqnarray}
The finite cover $\mu$  induces a finite cover $\phi:\Gamma:=X\times_{S^{(N)}} S^N\rightarrow X,\,\psi:\Gamma
\rightarrow S^{N}$ and we have a commutative diagram
\begin{eqnarray}\label{numerodiag111}
  \xymatrix{
&\Gamma\ar[r]^{\psi}\ar[d]^{\phi}&S^N\ar[d]^{\mu}\\
&X\ar[r]^{\phi_{Z_S}}&S^{(N)}}
\end{eqnarray}
From the commutativity of (\ref{numerodiag111}), we deduce the equality of holomorphic $2$- forms on $\Gamma$
\begin{eqnarray}\label{eqpour25oct}  \phi^*(\phi_{Z_S}^*\sigma_{S^{(N)}})=\psi^*(\mu^*\sigma_{S^{(N)}}).
\end{eqnarray}
Combining (\ref{eqnpourformetireeprod}) and
(\ref{eqpour25oct}), we get (\ref{eqdeprop}) with $\sigma_X=Z_S^*\sigma_S,\,\sigma_i=\sigma_S$ for all $i$.
\end{proof}

\begin{rema}{\rm To be fully rigorous in the above proof, we should introduce desingularizations of
$S^{(n)}$ and $\Gamma$ to write the equalities above. This is done  in \cite{mumford}.
}
\end{rema}

\subsection{Surface decomposition and cohomology ring}
It is a well-known and fundamental result (see \cite{bogover}, \cite{fujiki}) that
for a hyper-K\"ahler manifold $X$ of dimension $2n$, there exist  a quadratic
form $q$ on $H^2(X,\mathbb{Q})$ and  a positive  rational number $\lambda$ such that
for any $\eta\in H^2(X,\mathbb{Q})$
\begin{eqnarray}\label{eqintersection} \int_X\eta^{2n}=\lambda q(\eta)^n.
\end{eqnarray}
Let us show how this property  follows, at least on transcendental
cohomology,  from the existence of a surface decomposition for
a smooth  projective  variety  $X$, assuming it satisfies the following property (*). Recall  first  that
a quadratic form $q$ on a rational weight $2$   Hodge structure

$$H,\,H_\mathbb{C}=H^{2,0}\oplus H^{1,1}\oplus H^{0,2},\,\,\overline{H^{p,q}}=H^{q,p}$$
is said to satisfy the first Hodge-Riemann relations
if the Hodge decomposition is orthogonal
for the Hermitian pairing $h(\alpha,\beta)=q(\alpha,\overline{\beta})$ on
$H_\mathbb{C}$
or, equivalently, $q(H^{2,0},H^{2,0}\oplus H^{1,1})=0$. We will say that it satisfies the weak second Hodge-Riemann relations
if $q(\alpha,\overline{\alpha})\geq 0$ for $\alpha\in H^{2,0}$ and
 $q(\alpha,\overline{\alpha})\leq 0$ for $\alpha\in H^{1,1}$.
Consider the condition

\vspace{0.5cm}

(*)
{\it There exists up to a coefficient a unique quadratic form $q$
   satisfying  the first  Hodge-Riemann relations on $H^2(X,\mathbb{Q})_{tr}$.}

\vspace{0.5cm}

 Property (*) is well-known to be satisfied by a very general lattice polarized
 projective hyper-K\"ahler manifold. Note that we need in any case to use
  transcendental cohomology, namely
$H^2(X,\mathbb{Q})^{\perp NS(X)}$, instead of  primitive cohomology, as  (*)
is never satisfied on $H^2(X,\mathbb{Q})_{prim}$ if it is different from $H^2(X,\mathbb{Q})_{tr}$,
 that is, if it  contains
rational classes of type $(1,1)$.
We have the following.
\begin{prop}\label{propquadricfujiki}  (i) If a smooth projective variety  $X$ of
 dimension $2n$ admits  a surface decomposition as in Definition \ref{defiintrodecomp},
 there exist quadratic forms $q_1,\ldots,\,q_n$  satisfying the
 first and weak second Hodge-Riemann relations  on $H^2(X,\mathbb{Q})_{tr}$,
 such that,
 for any $\eta\in H^2(X,\mathbb{Q})_{tr}$
 \begin{eqnarray}\label{eqintersectionpourXgen} \int_X\eta^{2n}= q_1(\eta)\ldots q_n(\eta).
\end{eqnarray}

(ii) If furthermore $X$ satisfies property (*),
 there exists  a rational number $\lambda$ and a quadratic form $q$ satisfying the
 first and weak second  Hodge-Riemann relations  on $H^2(X,\mathbb{Q})_{tr}$, such that,
for any $\eta\in H^2(X,\mathbb{Q})_{tr}$,

\begin{eqnarray}\label{eqintersection2} \int_X\eta^{2n}=\lambda q(\eta)^n.
\end{eqnarray}
\end{prop}
\begin{proof} We have by assumption, for any $\eta\in H^2(X,\mathbb{Q})_{tr}$,  an equality
\begin{eqnarray}\label{eqencoretiregamma} \phi^*\eta=\psi^*(\sum_{i=1}^n {\rm pr}_i^*\eta_i),
\end{eqnarray}
where $\psi,\,\phi$ are as in (\ref{eqdecompholo}).
For each surface $S_i$, we have the Poincar\'e pairing $(\,,\,)_{S_i}$
on $H^2(S_i,\mathbb{Q})_{tr}$ which satisfies the first and second Hodge-Riemann relations, and, as the morphism $\psi_i$ which maps $\eta$ to $\eta_i$
 is a morphism of Hodge structures (see (\ref{eqpourpsii})), it
provides an intersection form $q_i(\eta):=(\eta_i,\eta_i)_{S_i}$ on $ H^2(X,\mathbb{Q})_{tr}$,
which satisfies the first and weak second Hodge-Riemann relations.

Let now $N,\,M$ be the respective  degrees of the maps $\phi,\,\psi$. We deduce from
(\ref{eqencoretiregamma}) the following equality:
\begin{eqnarray}\label{eqencoretiregammaproduit} N\int_X\eta^{2n}
=M\int_{S_1\times\ldots \times S_n}\sum_{i=1}^n({\rm pr}_i^*\eta_i)^{2n}=
M\frac{(2n)!}{2^nn!} (\eta_1,\eta_1)_{S_1}\ldots  (\eta_n,\eta_n)_{S_n}.
\end{eqnarray}
Let $q_1(\eta):=(\eta_1,\eta_1)_{S_1},\ldots,\,q_n(\eta):=(\eta_n,\eta_n)_{S_n}$, where
the $\eta_i$'s are defined by (\ref{eqencoretiregamma}). Then (\ref{eqencoretiregammaproduit})  gives (\ref{eqintersectionpourXgen}) up to  a multiplicative coefficient, which proves (i).

We next assume property (*) which implies  that $(\eta_i,\eta_i)_{S_i}=\mu_iq(\eta)$
for some rational numbers $\mu_i$, since  $q_i$ satisfies the first  Hodge-Riemann relations.
Equation (\ref{eqencoretiregammaproduit})  then  gives:
$$N\int_X\eta^{2n}=M\frac{(2n)!}{2^nn!} \mu_1\ldots\mu_n q(\eta)^n,$$
proving (ii).
\end{proof}

Proposition \ref{propquadricfujiki}  (i) now implies  the following result, showing that having a surface decomposition is a restrictive condition:
 \begin{theo} \label{coronotdecomp} Let
$S_1,\,S_2,\,S_3$ be three smooth projective  surfaces with $h^{1,0}(S_i)=0,\,h^{2,0}(S_i)\not=0$ for all $i$, and let
$H={\rm pr}_1^*H_1+{\rm pr}_2^*H_2+{\rm pr}_3^*H_3\in {\rm Pic}\,(S_1\times S_2\times S_3)$ be a very ample divisor on $S_1\times S_2\times S_3$. Let $Y\subset S_1\times S_2\times S_3$
 be the smooth complete intersection of two general members of $|H|$.
Then $Y$ is not surface decomposable.
\end{theo}
 \begin{proof} As $h^{1,0}(S_i)=0$, we have
$$H^2(S_1\times S_2\times S_3,\mathbb{Q})=H^2(S_1,\mathbb{Q})\oplus  H^2(S_2,\mathbb{Q})
\oplus  H^2(S_3,\mathbb{Q})$$
and similarly for transcendental cohomology.
 By Lefschetz hyperplane section theorem, we get, as ${\rm dim }\,Y=4$:
$$H^2(Y,\mathbb{Q})_{tr}=H^2(S_1,\mathbb{Q})_{tr}\oplus H^2(S_2,\mathbb{Q})_{tr}\oplus
  H^2(S_3,\mathbb{Q})_{tr},$$
 We now compute $\int_Y\alpha^4$ for $\alpha\in H^2(Y,\mathbb{Q})_{tr}$:
For $\alpha=\alpha_1+\alpha_2+\alpha_3$, using $\int_{S_i}\alpha_i\cup h_i=0$, where $h_i:=c_1(H_i)$, we get
 $$\int_Y\alpha^4=\int_{S_1\times S_2\times S_3}({\rm pr}_1^*\alpha_1+{\rm pr}_2^*\alpha_2+{\rm pr}_3^*\alpha_3)^4({\rm pr}_1^*h_1+{\rm pr}_2^*h_2+{\rm pr}_3^*h_3)^2$$

 \begin{eqnarray}
\label{eqavecnum}=\lambda_1
q_2(\alpha_2)q_3(\alpha_3)+\lambda_2q_1(\alpha_1)q_3(\alpha_3)+\lambda_3q_1(\alpha_1)q_2(\alpha_2),
\end{eqnarray}
where $q_i(\alpha_i):=\int_{S_i}\alpha_i^2$, and the constants $\lambda_i$ are nonzero
 rational numbers.
 It is immediate to see that (\ref{eqavecnum}) is not
of the form (\ref{eqintersectionpourXgen}), namely the product of two quadrics in
$\alpha=\alpha_1+\alpha_2+\alpha_3$. Indeed,
the hypersurface in
$\mathbb{P}(H^2(Y,\mathbb{C})_{tr})$ defined by (\ref{eqavecnum}) is irreducible,
being  fibered   with irreducible fibers  over the smooth conic in $\mathbb{P}^2_\mathbb{C}$
with equation $\lambda_1 y_2y_3+\lambda_2 y_1y_3+\lambda_3y_1y_2=0$, via the rational map
$$ \mathbb{P}(H^2(Y,\mathbb{C})_{tr})\dashrightarrow \mathbb{P}^2_\mathbb{C},$$
$$\alpha=\alpha_1+\alpha_2+\alpha_3\mapsto (q_1(\alpha_1),q_2(\alpha_2),q_3(\alpha_3)).$$
\end{proof}
\subsection{Application to Beauville weak splitting conjecture}
In the paper \cite{beauvoi}, it was observed that a projective $K3$ surface has the following
property:
there is a canonical $0$-cycle $o_S\in{\rm CH}_0(S)$ of degree $1$ (in fact, it can be defined
as $\frac{c_2(S)}{24}$)
such that for any divisor $D\in {\rm Pic}\,S={\rm CH}^1(S)={\rm NS}(S)$, one has
\begin{eqnarray}\label{eqchowBV} D^2=q(D)o_S\,\,{\rm in}\,\,{\rm CH}_0(S),
\end{eqnarray}
where $q(D)=([D],[D])_S$.
One can rephrase this result by saying that any cohomological polynomial relation
$$Q([D_1],\ldots,[D_k])=0\,\,{\rm in}\,\, H^*(S,\mathbb{Q})$$ involving
 only divisor classes is already
 satisfied in ${\rm CH}(S)_{\mathbb{Q}}$.
In \cite{beau}, Beauville made the following conjecture, generalizing the result above:
\begin{conj}\label{conjws} Let $X$ be a projective hyper-K\"ahler manifold. Then
the cycle class map is injective on the subalgebra of ${\rm CH}^*(X)_{\mathbb{Q}}$ generated by
divisor classes.
\end{conj}
Let us discuss Conjecture \ref{conjws} in relation with the notion of surface decomposition.
Let $X$ be a projective hyper-K\"ahler manifold, and let
$\Lambda\subset {\rm NS}(X)$ be a lattice polarization (which means that
$\Lambda$ contains an ample class. We assume that
the general $\Lambda$-polarized deformation $X_t$
(parameterized by  a quasiprojective basis $B$) of $X$ has a surface decomposition. Then, by standard
spreading arguments involving relative Chow varieties, after
passing to a generically finite cover $B'$ of $B$, we have
   projective morphisms $\Gamma\rightarrow B'$, $\mathcal{S}_i\rightarrow B'$,
   with ${\rm dim}\,\mathcal{S}_i/B'=2$,
   and morphisms over $B'$
  $$\phi:\Gamma\rightarrow \mathcal{X},\,\,\psi:\Gamma\rightarrow
  \mathcal{S}_1\times_{ B'}\ldots\times_{ B'}\mathcal{S}_n$$
  inducing a surface decomposition at the general point $t\in B'$.  After shrinking $B'$,
   by desingularization of the general fiber,  one can assume
 that the fibers $\Gamma_t$ and $S_{i,t}$ are  smooth and we get by specialization
 a diagram
\begin{eqnarray}
\label{eqphipsigenHK} \phi_t:\Gamma_t\rightarrow X_t,\,\,\psi_t:\Gamma_t
\rightarrow S_{1,t}\times\ldots\times S_{n,t}
\end{eqnarray}
  such that
\begin{eqnarray}\label{eqphipsigenHKforms}\phi_t^*\sigma_{X_t}=
\psi_t^*(\sum_i{\rm pr}_i^*\sigma_{S_{i,t}})\,\,{\rm in}\,\,H^{2,0}(\Gamma_t)
\end{eqnarray}
for some $(2,0)$-forms $\sigma_{S_{i,t}}$ on $S_{i,t}$.
We already observed that the relation (\ref{eqphipsigenHKforms}) in fact holds for any class
$\alpha\in H^2(X_t,\mathbb{Q})^{\perp\Lambda}$ and that
there is for each $i$ a (locally constant) morphism of Hodge structures
$$ \psi_{i,t}:H^2(X_t,\mathbb{Q})^{\perp\Lambda}\rightarrow H^2(S_{i,t},\mathbb{Q})$$
given by (\ref{eqpourpsii}) such that
\begin{eqnarray}\label{eqphipsigenHKclass}\phi_t^*\alpha=
\psi_t^*(\sum_i{\rm pr}_i^*(\psi_{i,t}(\alpha)))\,\,{\rm in}\,\,H^{2}(\Gamma_t,\mathbb{Q}).
\end{eqnarray}

Let us now  assume furthermore that $H^1(\Gamma_t,\mathbb{Z})=0$, or equivalently
\begin{eqnarray}\label{eqphipsigenHKformscondpic}{\rm NS}\,(\Gamma_t)={\rm Pic}\,(\Gamma_t).
\end{eqnarray}
In the situation described above, we have
 the following result.
\begin{theo}\label{prodecompws} For any  $t\in B'$, the weak splitting property
 holds  for divisor classes on $X_t$
which are in $H^2(X_t,\mathbb{Q})^{\perp\Lambda}$ if and only if, for each surface
$S_{i,t}$, the Beauville-Voisin relation (\ref{eqchowBV}) holds on ${\rm Im}\,\psi_{i,t}$ for an adequate $0$-cycle
$o_{S_{i,t}}\in{\rm CH}_0(S_{i,t})$.
\end{theo}
\begin{proof} Using (\ref{eqphipsigenHKformscondpic}), we conclude that
(\ref{eqphipsigenHKclass}) holds for $\alpha\in {\rm Pic}\,(X_t)^{\perp\Lambda}={\rm NS}\,(X_t)^{\perp\Lambda}$ (where the point $t$
is now special in $B'$, being in a Noether-Lefschetz locus),
and more precisely, that the morphism of Hodge structures
 $\psi_{i,t}$ induces for any $t\in B'$ a $\mathbb{Q}$-linear map
 $$\psi_{i,t}:{\rm Pic}\,(X_t)_{\mathbb{Q}}^{\perp\Lambda}\rightarrow {\rm Pic}\,(S_{i,t})_{\mathbb{Q}}$$
 such that,  for any $D\in {\rm Pic}\,(X_t)_{\mathbb{Q}}^{\perp\Lambda}$:
\begin{eqnarray} \label{eqphipsigenHKCH1}\phi_t^*D=
\psi_t^*(\sum_i{\rm pr}_i^*(\psi_{i,t}(D)))\,\,{\rm in}
\,\,{\rm Pic}(\Gamma_t)_{\mathbb{Q}}={\rm CH}^1(\Gamma_t)_{\mathbb{Q}}.
\end{eqnarray}
As in the cohomological setting  which  has been studied in the previous section, the important
point here is the fact that the pull-back maps appearing on both sides are  compatible
with intersection product. Note also that they are injective since the maps $\phi_t$ and $\psi_t$
are dominant.
For any point $t\in B$, let $D_1,\ldots,\, D_k\in {\rm CH}^1(X)_{\mathbb{Q}}$ and let
$Q$ be a degree $l$ homogeneous polynomial
with $\mathbb{Q}$-coefficients  in $k$ variables. Then we get from
(\ref{eqphipsigenHKCH1}):

\begin{eqnarray} \label{eqphipsigenHKCHl}\phi_t^*Q(D_1,\ldots,D_k)=
\psi_t^*(Q(D'_1,\ldots,\,D'_k))\,\,\,\,\,{\rm in}
\,\,{\rm CH}^l(\Gamma_t)_{\mathbb{Q}},\,\,\,{\rm where}\,\,\,\, D'_j:=\sum_i{\rm pr}_i^*(\psi_{i,t}(D_j)).
\end{eqnarray}
Assume that $X_t$ satisfies the weak splitting property, at least for divisor classes
$D\in{\rm CH}^1(X_t)^{\perp \Lambda}$.
There is then a $0$-cycle $o_X\in{\rm CH}_0(X)$ of degree
$1$ such that
\begin{eqnarray}\label{eqweakslit} D^{2n}=({\rm deg}\,D^n)o_X\,\,{\rm in}\,\,{\rm CH}_0(X)
\end{eqnarray}
for any $D\in {\rm CH}^1(X_t)^{\perp \Lambda}$.
Pulling-back to $\Gamma_t$ and using (\ref{eqphipsigenHKCHl}), we have now
\begin{eqnarray}\label{eqweakslitaveccomb}
\phi_t^*(D^{2n})=\frac{(2n)!}{2^nn!}\psi_t^*(\prod_{j=1}^n {\rm pr}_j^*(\psi_{j,t}(D)^2))
\,\,{\rm in}\,\,{\rm CH}_0(\Gamma_t).
\end{eqnarray}
Note that any $D\in {\rm CH}^1(X_t)^{\perp \Lambda}$ satisfies
$q(D)\not=0$ by the Hodge index theorem, where $q$ is the Beauville-Bogomolov quadratic form on
$H^2(X_t,\mathbb{Q})$, which can also be defined as the Lefschetz intersection
pairing on $H^2(X_t,\mathbb{Q})^{\perp\Lambda}$
(see \cite{beaujdg}). As we have ${\rm deg}\,D^{2n}=\lambda q([D])^n$ with $\lambda\not=0$  by
(\ref{eqintersection}), we conclude that  ${\rm deg}\,D^{2n}\not=0$.
Let $o_{S_{j,t}}:=
{\rm pr}_{j*}(\frac{1}{{\rm deg}\,\phi_t} (\psi_{t*}(\phi_t^* o_X)))
\,\,\in \,{\rm CH}_0(S_{j,t})_{\mathbb{Q}}$.
 This cycle
has degree $1$ and
we get from
(\ref{eqweakslitaveccomb}) by pushing-forward to $S_{j,t}$ via ${\rm pr}_j\circ \psi$ that
$\psi_{j,t}(D)^2$ is proportional to $o_{S_{j,t}}$. Indeed,
$({\rm pr}_j\circ \psi)_*(\phi_t^*(D^{2n}) )$ is a $0$-cycle of degree different from $0$ on $S_{j,t}$, which by
(\ref{eqweakslitaveccomb}) is proportional to both  $o_{S_{j,t}}$ and   $\psi_{j,t}(D)^2$.
This proves the ``only if'' direction.

Conversely, assume each surface $S_{i,t}$ has a $0$-cycle $o_{S_{i,t}}$ of degree $1$
with the property that divisors $D_i$ in ${\rm Im}\,\psi_{i,t}\subset {\rm NS}(S_{i,t})_\mathbb{Q}={\rm Pic}\,(S_{i,t})_\mathbb{Q}$
satisfy
${D}_i^2=( D_i,D_i)_{S_{i,t}} o_{S_{i,t}}$ in ${\rm CH}_0(S_{i,t})$ or equivalently
that for any  $D_i,\,D'_i \in {\rm Im}\,\psi_{i,t}$
\begin{eqnarray}
\label{eqouddrepddprime}{D}_i\cdot D'_i=(D_i,D'_i)_{S_{i,t}} o_{S_{i,t}}\,\,{\rm
 in }\,\,{\rm CH}_0(S_{i,t}).
\end{eqnarray}
We now use the fact that, at  the very general point of $B'$, the Mumford-Tate group
 of the Hodge structure
on $H^2(X_t,\mathbb{Q})^{\perp\Lambda}$ is the orthogonal group, and thus the
intersection form $\psi_{i,t}^*((\,,\,)_{S_{i,t}})$ equals
$\mu_i q$ on $H^2(X_t,\mathbb{Q})^{\perp\Lambda}$, for some coefficient $\mu_i$.
It then follows from (\ref{eqouddrepddprime}) that
a numerical  relation $q(D)=0$ for $D\in\,{\rm Pic}\,(X_t)_\mathbb{C}$ produces
relations
\begin{eqnarray}
\label{eqrelSi}  D_i^2=0\,\,{\rm in}\,{\rm CH}_0(S_{i,t})_\mathbb{C},
\end{eqnarray}
for any $i=1,\ldots, n$, where $D_i:=\psi_{i,t}(D)$.

By \cite{bogover}, we know that the
relations in the subalgebra of $H^*(X_t,\mathbb{C})$ generated by ${\rm Pic}\,(X_t)_\mathbb{C}$
are
generated by the Bogomolov-Verbitsky  relations
\begin{eqnarray}\label{eqbogover}
d^{n+1}=0\,\,{\rm if}\,\,q(d)=0.
\end{eqnarray} This is true as well if we restrict to
the subalgebra generated by ${\rm Pic}\,(X_t)^{\perp\Lambda}_\mathbb{C}$.
Finally
(\ref{eqphipsigenHKCHl}) provides for any $D\in {\rm Pic}\,(X_t)^{\perp\Lambda}_\mathbb{C}$
\begin{eqnarray}\label{eqdn+1}\phi_t^*( D^{n+1})=\psi_t^*((\sum_{i =1}^n{\rm pr}_i^* D_i)^{n+1})=
\sum_i{\rm pr}_1^*D_1\cdot\ldots \cdot {\rm pr}_i^*D_i^2\cdot {\rm pr}_n^* D_n\\
\nonumber
+\ldots\,\,
{\rm in}\,\,{\rm CH}(\Gamma_t)_\mathbb{C},
\end{eqnarray}
where the remaining term ``$\ldots$''  involves products ${\rm pr}_i^*D_i^2\cdot {\rm pr}_j^*D_j^2$
 of two squares, then three squares ${\rm pr}_i^*D_i^2\cdot {\rm pr}_j^*D_j^2\cdot  {\rm pr}_k^*D_k^2$ etc...
Using (\ref{eqrelSi}) and (\ref{eqdn+1}), we get $\phi_t^*(D^{n+1})=0$
in ${\rm CH}^{n+1}(\Gamma_t)_\mathbb{C}$, hence
$D^{n+1}=0$ in ${\rm CH}^{n+1}(X_t)_\mathbb{C}$, whenever $q(D)=0$. In other words,  the Bogomolov-Verbitsky
 relations (\ref{eqbogover}) are satisfied in
${\rm CH}^{n+1}(X_t)_\mathbb{C}$, which concludes the proof.
\end{proof}
We get the following corollary:
\begin{coro} (Cf. Theorem \ref{theoconsdeprodecompws}) Under the same assumptions
as in Theorem \ref{prodecompws}, the weak splitting property holds  for divisor classes on $X_t$
which are in $H^2(X_t,\mathbb{Q})^{\perp\Lambda}$ if and only if they hold in top degree, that is,

(*)  there exists
a canonical $0$-cycle $o_{X_t}\in{\rm CH}_0(X_t)$ such that for any $D\in {\rm NS}(X_t)^{\perp \Lambda}$,
$D^{2n}$ is proportional to  $o_{X_t}$ in ${\rm CH}_0(X_t)$.
\end{coro}
\begin{proof} The ``only if'' is clear. In the other direction, examining the proof of Theorem
\ref{prodecompws}, we observe that we only used relations (\ref{eqweakslitaveccomb})  in top degree $2n$
 to conclude that, if (*) holds,  defining
$o_{S_{i,t}}\in{\rm CH}_0(S_{i,t}):=\frac{1}{{\rm deg}\,\phi_t}{\rm pr}_{i*}(\psi_{t*}(\phi_t^*o_{X_t})]\in
 {\rm CH}_0(S_{i,t})$, the zero-cycle
$D_{i,t}^2$ is proportional to $o_{S_{i,t}}$ in  ${\rm CH}_0(S_{i,t})$,
for any $D_t\in{\rm NS}(X_t)^{\perp\Lambda}$,
where $D_{i,t}:=\psi_{i,t}(D_t)$.
Hence by Theorem
\ref{prodecompws}, (*) implies the  weak splitting property  for ${\rm NS}(X_t)^{\perp\Lambda}$.
\end{proof}
\subsection{Curve decompositions}
We can of course introduce decomposition into summands of other dimensions. For example,
we can consider {\it curve decompositions} of any variety $X$ of dimension $n$, given by the
data of generically finite surjective morphisms
\begin{eqnarray}
\label{eqphipsicurve} \phi:\Gamma\rightarrow X,\,\,\psi:\Gamma\rightarrow C_1\times\ldots\times C_n,
\end{eqnarray}
such that for any
$1$-form $\alpha\in H^{1,0}(X)$,
\begin{eqnarray}
\label{eqpourcuirvedec}\phi^*\alpha=\psi^*(\sum_i{\rm pr}_i^*\alpha_i)
\end{eqnarray}
for some forms $\alpha_i\in H^{1,0}(C_i)$.  Jacobians of curves
are clearly curve decomposable. It  might be interesting
to study how restrictive is  this notion.
We have the  obvious analogs of  Proposition \ref{propquadricfujiki}
 and Theorem \ref{coronotdecomp}. Recall that the first Hodge-Riemann relations
 for a skew-symmetric intersection pairing $\omega\in \bigwedge^2H^*$ on a weight $1$ rational Hodge structure
$$H,\,H_\mathbb{C}=H^{1,0}\oplus H^{0,1},\, \,H^{0,1}=\overline{H^{1,0}}$$
say that $\omega_{\mid H^{1,0}}=0$. The weak second Hodge-Riemann relations
say that $\omega(\alpha,\overline{\alpha})\geq 0$ for $\alpha\in H^{1,0}$.
\begin{prop}\label{propinterdecompcurve}  If $X$ admits  a curve  decomposition, there exist a positive rational number
$\lambda$ and
  skew pairings $\omega_i$  satisfying the first and weak second Hodge-Riemann relations
 on $H^1(X,\mathbb{Q})$   such that,
for any $\alpha_1,\ldots,\alpha_{2n}\in H^1(X,\mathbb{Q})$,

\begin{eqnarray}\label{eqintersection3} \int_X \alpha_1\cup
 \ldots\cup\alpha_{2n}=\lambda \sum_{P}\epsilon_P
\prod_{p\in P}\omega_i(\alpha_{p_i},\alpha_{p'_i}) .
\end{eqnarray}
Here $P$ runs through the set of partitions of $\{1,\ldots,2n\}$ into $n$  pairs
$(p_{i},p'_i)$ with $p_i<p'_i$ and $\epsilon_P$ is an adequate sign.
\end{prop}
\begin{proof} We have by assumption a decomposition
\begin{eqnarray}\label{eqencoretiregammacurve} \phi^*\alpha_j=\psi^*(\sum_{i=1}^n{\rm pr}_i^*\alpha_{ji}),
\end{eqnarray}
where $\psi,\,\phi$ are as in (\ref{eqphipsicurve}) and $\alpha_{ji}\in H^1(C_i,\mathbb{Q})$.
For each curve $C_i$, we have the Poincar\'e pairing $(\,,\,)_{C_i}$
on $H^1(C_i,\mathbb{Q})$ and as the maps $\alpha\mapsto \alpha_i$ appearing in
(\ref{eqpourcuirvedec}) are  morphisms of Hodge structures, this provides
the desired pairings   $\omega_i(\alpha,\beta):=(\alpha_i,\beta_i)_{C_i}$  satisfying the first and weak second Hodge-Riemann relations on $ H^1(X,\mathbb{Q})$.
Let now $N,\,M$ be the respective  degrees of the maps $\phi,\,\psi$. We deduce from
(\ref{eqencoretiregammacurve}) the following equality:
\begin{eqnarray}\label{eqencoretiregammaproduitcurve} N\int_X \alpha_1\cup \ldots\cup\alpha_{2n}=M\int_{C_1\times\ldots \times C_n}\prod_{j=1}^{2n}\sum_{i=1}^n{\rm pr}_i^*\alpha_{ji}=\\
\nonumber
M\sum_{P=\{p_i<p'_i\}_i}\epsilon_P \prod_{i=1}^n\int_{C_i}\alpha_{p_i,i}\wedge \alpha_{p'_i,i}
=M\sum_P\epsilon_P \prod_{i=1}^n\omega_i(\alpha_{p_i}, \alpha_{p'_i}),
\end{eqnarray}
where $\epsilon_P$ is the signature of the permutation
$\{1,\ldots,2n\}\rightarrow \{p_1,p'_1,\ldots,\,p_n,p'_n\}$.
\end{proof}

A clearer formulation of (\ref{eqintersection3}) is the following:
view the intersection form
$$\alpha_1\wedge\ldots \wedge\alpha_{2n}\mapsto \int_X\alpha_1\cup\ldots\cup\alpha_{2n}$$
on $H^1(X,\mathbb{Q})$ as an element $f_X$ of
$\bigwedge^{2n}H^1(X,\mathbb{Q})^*$, while the intersection pairings $\omega_i$
 belong to $\bigwedge^2H^1(X,\mathbb{Q})^*$.
Then   (\ref{eqintersection3}) simply says that
\begin{eqnarray}
\label{eqclearerformula} f_X=\frac{M}{N}\omega_1\wedge\ldots \wedge \omega_n\,\,{\rm in}\,\,\bigwedge^{2n}H^1(X,\mathbb{Q})^*.
\end{eqnarray}
As a consequence of Proposition \ref{propinterdecompcurve}, we now get the following statement analogous to
Theorem \ref{coronotdecomp}:
\begin{theo} \label{coronotcurvedecomp} Let $g\geq 13$ and let
 $C_1,\,C_2,\,C_3$ be three curves with $h^{1,0}(C_i)=g$ for all $i$, and let
$H={\rm pr}_1^*H_1+{\rm pr}_2^*H_2+{\rm pr}_3^*H_3\in {\rm Pic}\,(C_1\times C_2\times C_3)$ be a very ample divisor on $C_1\times C_2\times C_3$. Let $Y\subset C_1\times C_2\times C_3$
 be a smooth member of $|H|$.
 Then $Y$ is not curve decomposable.
 \end{theo}

 \begin{proof} By Lefschetz hyperplane section theorem, we get, as ${\rm dim }\,Y=2$:
$$H^1(Y,\mathbb{Z})=H^1(C_1,\mathbb{Z})\oplus H^1(C_2,\mathbb{Z})\oplus  H^1(C_3,\mathbb{Z}).$$

 We now compute $\int_Y\alpha_1\cup\alpha_2\cup \alpha_3\cup \alpha_4$ for $\alpha_i\in
 H^1(Y,\mathbb{Q})$.
 Writing $\alpha_i={\rm pr}_1^*\beta_{i1}+{\rm pr}_2^*\beta_{i2}+{\rm pr}_3^*\beta_{i3}$,  and denoting
 $h_i=c_1(H_i)\in H^2(C_i,\mathbb{Q})$, we get

 $$\int_Y\alpha_1\cup\alpha_2\cup \alpha_3\cup \alpha_4=\int_{C_1\times C_2\times C_3}
 \cup_{i=1}^{i=4}({\rm pr}_1^*\beta_{i1}+{\rm pr}_2^*\beta_{i2}+{\rm pr}_3^*\beta_{i3})\cup({\rm pr}_1^*h_1+{\rm pr}_2^*h_2+{\rm pr}_3^*h_3)$$

 \begin{eqnarray}
\label{eqavecnumcurve}=\lambda_1\sum_{i<j}\omega_{2}(\alpha_i,\alpha_j)
\omega_{3}(\alpha_k,\alpha_l)+\lambda_2\sum_{i<j}\omega_1(\alpha_i,\alpha_j)
\omega_3(\alpha_k,\alpha_l)+\lambda_3\sum_{i<j}\omega_1(\alpha_i,\alpha_j)\omega_2(\alpha_k,\alpha_l),
\end{eqnarray}

 where
  \begin{enumerate}
  \item the constants $\lambda_i$ are nonzero rational numbers,
   \item in each term, $k<l$ and $\{1,2,3,4 \}=\{i,j,k,l\}$,
    \item  for $l=1,\,2,\,3$, $\omega_l(\alpha_i,\alpha_j):=\int_{C_l}\beta_{il}\cup\beta_{jl}$.
\end{enumerate}
Said differently, we have
\begin{eqnarray}
\label{eqdecrivantintY}f_Y=\lambda_1\omega_1\wedge\omega_2+\lambda_2\omega_2\wedge\omega_3+
\lambda_3\omega_1\wedge\omega_3
\,\,{\rm in}\,\,\bigwedge^4H^1(Y,\mathbb{Q})^*.
\end{eqnarray}
If $Y$ was curve decomposable, then,
by Proposition \ref{propinterdecompcurve}, using the reformulation (\ref{eqdecrivantintY}), there would exist
$\omega',\,\omega''\in \bigwedge^2H^1(Y,\mathbb{Q})^*$ such that
\begin{eqnarray}
\label{eqpourprouduitsiexiste} f_Y=\omega'\wedge \omega'' \,\,{\rm in}\,\,\bigwedge^4H^1(Y,\mathbb{Q})^*.
\end{eqnarray}

Theorem  \ref{coronotcurvedecomp}  thus follows from the following lemma.
\begin{lemm}\label{lepourcurvedec} Let $g\geq 13$, and let $V_i$, $i=1,\,2,\,3$ be  three spaces of dimension $2g$ equipped with a nondegenerate skew pairing
$\omega_i\in \bigwedge^2V_i^*$. Let $V:=V_1\oplus V_2\oplus V_3$. We see
$\omega_i$ as elements of $\bigwedge^2V^*$. Then there do not exist $\omega,\,\omega'\in \bigwedge^2V^*$ such that
\begin{eqnarray}
\label{eqpourprouduitabstrait}
\omega_1\wedge\omega_2+\omega_2\wedge\omega_3+\omega_1\wedge\omega_3=\omega'\wedge \omega'' \,\,{\rm in}\,\,\bigwedge^4V^*.
\end{eqnarray}
\end{lemm}
We apply indeed Lemma \ref{lepourcurvedec} to $V_i:=H^1(C_i,\mathbb{Q}),\,\omega_i=(\,,\,)_{C_i}$. Using
(\ref{eqdecrivantintY}), Lemma \ref{lepourcurvedec} says that
 (\ref{eqpourprouduitsiexiste}) cannot hold.
\end{proof}
\begin{proof}[Proof of Lemma \ref{lepourcurvedec}] Let $W_1\subset V_1,\,W_2\subset V_2$ be generic Lagrangian subspaces with respect to
$\omega_1$, resp. $\omega_2$. Let
$K:=W_1\oplus W_2\oplus V_3$. Then $\omega_1\wedge\omega_2+\omega_2\wedge\omega_3+\omega_1\wedge\omega_3$ vanishes on $K$. If (\ref{eqpourprouduitabstrait}) holds, then
$\omega'\wedge\omega''$ vanishes on $K$.

We now have the following lemma.
\begin{lemm}\label{lemmeutilesurproduit} Let $V$ be a vector space and let
 $\omega,\,\omega'\in \bigwedge^2V^*$ such that $\omega\not=0,\,\omega'\not=0$
 and  $\omega\wedge\omega'=0$ in $ \bigwedge^4V^*$. Then there exists a quotient
 $V\rightarrow V'$ with ${\rm dim}\,V'\leq 4$  such that both
 $\omega$ and $\omega'$ are pulled back from $2$-forms on $V'$.
\end{lemm}
\begin{proof} This follows from the fact that for a  $2$-form $\omega$ of rank $6$ on
a $6$-dimensional vector space $V$, the wedge product map
$\omega\wedge: \bigwedge^2{V}^*\rightarrow \bigwedge^4{V}^*$ is an isomorphism.
This fact already implies  that if $\omega\wedge\omega'=0$ in $ \bigwedge^4V^*$, with
 $\omega\not=0,\,\omega'\not=0$, the rank of
$\omega$ is at most $4$ and similarly for $\omega'$. If both forms $\omega$ and $\omega'$
are of rank $2$, the conclusion
of the lemma holds. If $\omega=e_1^*\wedge e_2^*+e_3^*\wedge e_4^*$ is of rank $4$, let
$V'=\langle e_1^*,\ldots,\,e_4^*\rangle$. Choosing a decomposition $V^*={V'}^*\oplus W^*$, we can write
$\omega'=\alpha+\beta+\gamma$ with $$\alpha\in \bigwedge^2{V'}^*,\,\,
\beta\in {V'}^*\otimes W^*,\,\,\gamma\in\bigwedge^2W^*,$$
and we must have $\omega\wedge\beta=0,\,\omega\wedge \gamma=0$, which clearly implies
that $\beta=0$ and $\gamma=0$ because $\omega$ has rank $4$  so $\omega\wedge $ is injective on
${V'}^*$. Thus
$\omega'$  belongs to $\bigwedge^2{V'}^*$.
\end{proof}
Coming back to our situation, Lemma
\ref{lemmeutilesurproduit} shows that are  two possibilities:
either (1) one of  $\omega'$, $\omega''$ vanishes on $K$, or
\begin{eqnarray}\label{eqranks}
{\rm rank}\,\omega'_{\mid K}\leq 4,\,\,{\rm rank}\,\omega''_{\mid K}\leq 4.
\end{eqnarray}
In the situation (1), $\omega'$ vanishes on $K$ for a general choice of $W_1,\,W_2$, and this easily implies that
$V_3$ is contained in the kernel of $\omega'$. But then neither $V_1$, nor $V_2$ can be contained in
the kernel of $\omega'$, since otherwise $\omega'\wedge \omega''$ would vanish on $V_1\oplus V_3$ or
$V_2\oplus V_3$, and this is not the case.
Similarly $\omega''$ can have at most one of $V_1,\,V_2,\,V_3$ contained in its kernel. Hence permuting the $V_i$'s if necessary,
we can assume that $V_3$ is not contained in the kernel of $\omega'$ nor $\omega''$, hence that
neither $\omega'$ nor $\omega''$ vanishes on $K$.

We thus excluded (1) and can assume that (\ref{eqranks}) holds.
Consider the map
$$\alpha'_1: V_3\rightarrow V_1^*,\, \alpha'_2: V_3\rightarrow V_2^*$$
obtained by composing
$\lrcorner\omega': V_3\rightarrow V_1^*\oplus V_2^*\oplus V_3^*$
with the projection to $V_1^*$, resp. $V_2^*$. For $i=1,\,2$, let $ K^*_{i,\omega'}\subset V_i^*$ be its image. One easily checks that
the restriction map
$$K^*_{i,\omega'}\rightarrow W_i^*$$
has maximal rank for generic choices of Lagrangian spaces  $W_i$, that is,
it is either injective or surjective. As it has rank
$\leq 4$ by (\ref{eqranks}), and $g\geq 13$, one concludes that ${\rm dim}\,K_{i,\omega'}^*\leq 4$.
It follows that
the map $\alpha'_1+\alpha'_2:V_3\rightarrow V_1^*\oplus V_2^*$  has rank $\leq 8$
 and as $\omega'_{\mid V_3}$
has rank $\leq 4$, one concludes
that the rank of
the map
$\lrcorner\omega'_{\mid V_3}: V_3\rightarrow V_1^*\oplus V_2^*\oplus V_3^*$
 is at most $12$. We argue similarly for
$\lrcorner\omega''_{\mid V_3}: V_3\rightarrow V_1^*\oplus V_2^*\oplus V_3^*$
As ${\rm dim}\,V_3=2g\geq 26$, we conclude that
$${\rm Ker}\,\lrcorner\omega'_{\mid V_3}\cap {\rm Ker}\,\lrcorner\omega''_{\mid V_3}\not=\{0\}.$$
But this  contradicts formula (\ref{eqpourprouduitabstrait}). Indeed, this implies that the right hand side has a nontrivial kernel, that is, is pulled-back from some proper quotient of $V_1\oplus V_2\oplus V_3$, while formula on the left is an alternating
$4$-form with no nontrivial kernel.
\end{proof}

\subsection{Decomposition from families of algebraically coisotropic divisors \label{secdecdiv}}
We study in this section a weaker notion of decomposition for a holomorphic $2$-form
 into forms of smaller rank  (see Question \ref{question01}). The following is a weak converse to
 Proposition \ref{propeasysenseintro}.

\begin{prop}\label{theodecompdiv}  Let $X$ be smooth projective variety of dimension $2n$ equiped with a holomorphic $2$-form
$\sigma_X$.
Assume $X$ is swept-out by (possibly singular) algebraically coisotropic divisors. Then there exists a generically
finite cover
$\Phi: \mathcal{D}'\rightarrow X$
such that
\begin{eqnarray}\Phi^*\sigma_X=\eta_1+\eta_2\,\,{\rm in }\,\,H^{2,0}(\mathcal{D}'),
\end{eqnarray}
where ${\rm rank}\,\eta_1\leq2$,  and  ${\rm rank}\,\eta_2\leq 2n-2$. More precisely,
$\eta_2$ is the  pull-back of a holomorphic $2$-form on
a  variety of dimension $\leq 2n-1$.
\end{prop}
Here by the rank we mean the generic rank of the considered forms.
\begin{proof}[Proof of Proposition  \ref{theodecompdiv}] By assumption, there exists a
$1$-parameter family
$$\mathcal{D}\rightarrow C,\,\,\mathcal{D}\rightarrow X$$ of  divisors
$D_t\subset X$ whose  characteristic foliation (on the regular locus of $D_t$)
 is algebraically integrable, that is, there exists a rational map
 $$\phi_t:D_t\dashrightarrow B_t$$
 with ${\rm dim}\,B_t=2n-2$, such that the equality
 $\sigma_{X\mid D_t}=\phi_t^*\sigma_{B_t}$,  for some holomorphic $2$-form $\sigma_{B_t}$
 on the regular locus
 of $B_t$, holds
 on the regular locus of $D_t$.
 Note that, by desingularisation, we can assume $D_t$ and  $B_t$ smooth, at least for general $t$.
  Indeed, the $2$-form
   $\sigma_{B_t}$
 extends  holomorphically  on any smooth projective  model $\widetilde{B}_t$ of $B_t$, because
  it can be constructed as
 $$\tilde{\phi}_{t*}(\tilde{j}_t^*\sigma_{X}\wedge \omega)$$
 where
 $\tilde{j}_t:\widetilde{D}_t\rightarrow X$ is a smooth model  of $D_t$
 such that $\tilde{\phi}_{t}:\widetilde{D}_t\rightarrow\widetilde{B}_t$ is a  morphism, and $\omega$ is a closed $(1,1)$-form
 on $\widetilde{D}_t$ whose integral over the fibers of $\tilde{\phi}_{t}$ is $1$.

As usual, the data above (namely the family of varieties $B_t$ and morphisms
 $\phi_t$) can be put in family, possibly after base change
 from the original family
 $\mathcal{D}\rightarrow C$ of divisors on $X$ and birational transformations. We thus get
 the following diagram

 \begin{eqnarray}
  \xymatrix{
&\mathcal{D}'\ar[r]^{J}\ar[d]^{\Phi}&X\\
&B&}
\end{eqnarray}
where all the varieties are smooth and projective, the morphism $J$ is surjective generically
finite and
$B$ admits a morphism
$f:B\rightarrow C$ such that, considering  the induced
diagram of fibers over a general point  $t\in C$
 \begin{eqnarray}\label{numerodiagt}
  \xymatrix{
&\mathcal{D}'_t\ar[r]^{J_t}\ar[d]^{\Phi_t}&X\\
&B_t&},
\end{eqnarray}
one has
\begin{eqnarray}\label{eqrestsigmaBJphi}J_t^*\sigma_X=\Phi_t^*\sigma_{B_t}\,\,{\rm in}\,\,H^{2,0}(B_t).
\end{eqnarray}
We deduce from this last equality that
the forms $\sigma_{B_t},\,t\in C$, form a locally constant section of
the bundle $\mathcal{H}^{2,0}\subset R^2f_*\mathbb{C}\otimes \mathcal{O}_C$
on the open set of $C$ of regular values of $f$.
By the global invariant cycles theorem
\cite{deligne}, \cite[4.3.3]{voisinbook}, there exists a holomorphic $2$-form
$\sigma_B\in H^{2,0}(B)$ such that
\begin{eqnarray}\label{eqrestsigmaB}\sigma_{B\mid B_t}=\sigma_{B_t}.
\end{eqnarray}
We conclude from (\ref{eqrestsigmaBJphi}) and (\ref{eqrestsigmaB}) that
the $2$-form $\Phi^*\sigma_B-J^*\sigma_X$ vanishes on the divisors
$\mathcal{D}'_t=\Phi^{-1}(B_t)$ which cover $\mathcal{D}'$. This form thus has rank $\leq 2$ on
 $\mathcal{D}'$.
\end{proof}
This statement rises the following question.

\begin{question}\label{qdiv} Is any projective  hyper-K\"ahler manifold  swept out by
algebraically coisotropic divisors?

\end{question}
 The following question was asked by G. Pacienza.
 \begin{question}\label{qpac} Is any projective hyper-K\"ahler manifold
  swept out by elliptic curves?
  \end{question}
  The following proposition relates Question \ref{qdiv} and Question
  \ref{qpac}.
  \begin{prop} If a very general polarized hyper-K\"ahler manifold with $b_2\geq 5$
   is swept out by elliptic curves, then it is swept out by
algebraically coisotropic divisors.
  \end{prop}
  Here, ``very general" means that $X$ is the very general member
  of a complete family of polarized hyper-K\"ahler manifolds.
  \begin{proof} There exists by assumption  a covering  family of elliptic curves
  $$\phi: \mathcal{E}\rightarrow X,\,\psi:\mathcal{E}\rightarrow B$$
  with $\phi$ surjective  generically finite and  ${\rm dim}\,B=2n-1$.
  If these elliptic curves have constant moduli, after passing to a generically finite cover of $B$,
  $\mathcal{E}$ becomes birational to a product $E\times B$ and
  we conclude that
  there is an injective morphism
  of Hodge structures
  $$H^2(X,\mathbb{Q})_{tr}\rightarrow H^1(E,\mathbb{Q})\otimes H^1(B,\mathbb{Q}).$$
   Indeed, $\phi^*\sigma_X$ is not in the image of $\psi^*$ because $\psi^*H^{2,0}(B)$
  consists of holomorphic forms of generic rank $<{\rm dim}\,X$, while $\phi^*\sigma_X$
  has generic rank equal to ${\rm dim}\,X$. Hence $\phi^*\sigma_X$ has a nontrivial
  image in $H^1(E,\mathbb{C})\otimes H^1(B,\mathbb{C})$. The natural morphism
  $H^2(X,\mathbb{Q})_{tr}\rightarrow H^1(E,\mathbb{Q})\otimes H^1(B,\mathbb{Q})$ given by pull-back
  and projection to a Leray summand
  is thus nonzero on $H^{2,0}(X)$, hence injective on $H^2(X,\mathbb{Q})_{tr}$.
  When $b_2\geq 5$ and $X$ is very general, the existence of such injective morphism
  contradicts the result of \cite{geemenvoisin}. Hence the elliptic curves $E$ must have
  variable modulus.
  For a fixed $t\in \mathbb{P}^1$, consider the divisor $B_t\subset $
  parameterizing   elliptic curves $E_b$ with fixed $j$-invariant determined by
  $t$. Over $B_t$, the family $\mathcal{E}_t=\psi^{-1}(B_t)$ is birational (possibly after
  after base change) to $E_t\times B_t$. Let
  $$\phi_t: \mathcal{E}_t\rightarrow X,\,\psi_t:\mathcal{E}_t\rightarrow B_t$$
  be the restricted family. The same argument as above shows that
  $\phi_t^*\sigma_X$ has to vanish in $H^1(E_t,\mathbb{C})\otimes H^1(B_t,\mathbb{C})$.
This is exactly saying that $\phi(\mathcal{E}_t)$ is an algebraically coisotropic divisor in $X$,
as this implies that $\phi_t^*\sigma_X$ is pulled-back from $B_t$.
  \end{proof}
 It seems plausible that Question \ref{qpac} has  a negative answer while Question
\ref{qdiv} has a positive answer.
\section{Triangle varieties: examples \label{secexamples}}
Recall from Definition \ref{defitriangle} in the introduction that a triangle variety
$T$ for a hyper-K\"ahler manifold $X$ of dimension $2n$ is a subvariety of $X\times X\times X$ which has dimension
$3n$, maps surjectively onto the various summands and maps in a generically finite way on its image
in the product of two summands, and
is such that
\begin{eqnarray}\label{eqtriangletexte}
({\rm pr}_1^*\sigma_X+{\rm pr}_2^*\sigma_X+{\rm pr}_3^*\sigma_X)_{\mid T_{{\rm reg}}}=0,
\end{eqnarray}
where $\sigma_X$ is the holomorphic $2$-form of $X$.
 Note that Equation (\ref{eqtriangletexte}) says that $T$ is Lagrangian for
the everywhere nondegenerate holomorphic $2$-form ${\rm pr}_1^*\sigma_X+{\rm pr}_2^*\sigma_X+{\rm pr}_3^*\sigma_X$
on $X^3$. A variant of the main deformation invariance theorem
of \cite{voisinlag} says now the following:
\begin{theo} \label{theolagco} Let $X$ be hyper-K\"ahler and
let $j: L\hookrightarrow  X\times X\times X$ be a smooth triangle subvariety
(hence $L$ is  Lagrangian  for the $2$-form ${\rm pr}_1^*\sigma_X+{\rm pr}_2^*\sigma_X+{\rm pr}_3^*\sigma_X$).
Then   for a small deformation $X_t$ of $X$ with constant Picard group, there is a deformation
$j_t:L_t\hookrightarrow  X_t\times X_t\times X_t$  of  $L$, and $L_t$ is  a triangle variety for $X_t$.
\end{theo}
The last statement follows from the fact
that, denoting $\Lambda={\rm NS}(X)$,  the
subgroup $H^2(X_t,\mathbb{Q})^{\perp \Lambda}$
and the  restriction map $$j_t^* :H^2(X_t\times X_t\times X_t,\mathbb{Q})\rightarrow H^2(L_t,\mathbb{Q})$$
are  locally constant on the base $B$ of deformations of $X$ with fixed Picard number.
Hence the diagonal image of $H^2(X_t,\mathbb{Q})^{\perp\Lambda}$ in
$H^2(X_t\times X_t\times X_t,\mathbb{Q})=H^2(X_t,\mathbb{Q})^3$ is annihilated by $j_t^*$, since it is annihilated by $j^*$
(note here that $H^2(X,\mathbb{Q})^{\perp \Lambda}=H^2(X,\mathbb{Q})_{tr}$).
\begin{rema}{\rm A smooth  triangle subvariety $L\subset X\times X\times X$ cannot deform in
products $X_t\times X_{t'}\times X_{t''}$ unless $t=t'=t''$. Indeed, the kernel $H$ of
$j^*:H^2(X,\mathbb{Q})_{tr}^3\rightarrow H^2(L,\mathbb{Q})$ is exactly the diagonal image of
$H^2(X,\mathbb{Q})_{tr}$, as it follows from the fact that $L$ maps to a subvariety of dimension $3n$ in the
three products $X\times X$. If there is
a deformation $L_{t,t',t''}$ of $L$ in
$X_t\times X_{t'}\times X_{t''}$, there is a Hodge substructure
$$H_{t,t',t''}\subset  H^2(X_t,\mathbb{Q})^{\perp\Lambda}\oplus  H^2(X_{t'},\mathbb{Q})^{\perp\Lambda}\oplus
H^2(X_{t''},\mathbb{Q})^{\perp\Lambda}$$
deforming $H$. But then $H_{t,t',t''}$ is isomorphic  by projections to the three Hodge structures
$H^2(X_t,\mathbb{Q})^{\perp\Lambda},\,\,  H^2(X_{t'},\mathbb{Q})^{\perp\Lambda},\,\,
H^2(X_{t''},\mathbb{Q})^{\perp\Lambda}$. By the local Torelli theorem, we then have $t=t'=t''$.
 }
\end{rema}
Theorem \ref{theolagco} suggests possibly  that triangle subvarieties tend to be stable under deformations
with constant Picard number, but in the examples we will describe below, the triangle
subvarieties are never smooth, so in fact
Theorem \ref{theolagco} does not apply.

Considering  the conjectures made in \cite{beau}, \cite{voicoiso}, it would be very nice that the triangle varieties
$T$ satisfy
  a cycle-theoretic variant of
(\ref{eqtriangletexte}), asking
the following:
for any $t=(t_1,t_2,t_3)\in T\subset X^3$
\begin{eqnarray}\label{eqpourcycle}  t_1+t_2+t_3 =c\,\,{\rm in}\,\,{\rm CH}_0(X),
\end{eqnarray}
for some  fixed zero-cycle $c$ of $X$.
Formula (\ref{eqpourcycle}) implies indeed (\ref{eqtriangletexte}) by Mumford's theorem \cite{mumford}.
Let us explain why it is not possible to achieve   (\ref{eqpourcycle}) starting from dimension $4$.

\begin{prop} \label{pronexistepasdanschow} Let $X$ be a projective  hyper-K\"ahler manifold of dimension $2n\geq 4$.
Let $T$ be a triangle subvariety of  $X\times X\times X$. Then the cycle $t_1+t_2+t_3\in {\rm CH}_0(X)$ for
$t= (t_1,t_2,t_3)\in T$ is not constant along
$T$.
\end{prop}
\begin{proof} Indeed, if  (\ref{eqpourcycle}) holds, then Mumford's theorem
\cite{mumford}  says that
for any power
$\sigma_X^l$, $l>0$,  of $\sigma_X$,
\begin{eqnarray}\label{eqpourpower}  ({\rm pr}_1^*\sigma_X^l+{\rm pr}_2^*\sigma_X^l+{\rm pr}_3^*\sigma_X^l)_{\mid T_{{\rm reg}}}=0\,\,{\rm in}\,\,H^0(T_{\rm reg},\Omega_{T_{\rm reg}}^{2l}).
\end{eqnarray}
We now set $l=2$.  We then have the two equations
\begin{eqnarray}\label{eqpourpowerdeuxeq}
 ({\rm pr}_1^*\sigma_X)_{\mid T_{{\rm reg}}}=-({\rm pr}_2^*\sigma_X+{\rm pr}_3^*\sigma_X)_{\mid T_{{\rm reg}}} \,\,{\rm in}\,\,H^0(T_{\rm reg},\Omega_{T_{\rm reg}}^{2}),
\\
\nonumber
({\rm pr}_1^*\sigma_X^2)_{\mid T_{{\rm reg}}}=-({\rm pr}_2^*\sigma_X^2+{\rm pr}_3^*\sigma_X^2)_{\mid T_{{\rm reg}}}\,\,{\rm in}\,\,H^0(T_{\rm reg},\Omega_{T_{\rm reg}}^{4}).
\end{eqnarray}
It follows that
\begin{eqnarray}\label{eqpourcrossproduct}
 -({\rm pr}_2^*\sigma_X^2+{\rm pr}_3^*\sigma_X^2)_{\mid T_{{\rm reg}}}=({\rm pr}_2^*\sigma_X+{\rm pr}_3^*\sigma_X)^2_{\mid T_{{\rm reg}}} \,\,{\rm in}\,\,H^0(T_{\rm reg},\Omega_{T_{\rm reg}}^{4}).
\end{eqnarray}
Let us write the above equation as
\begin{eqnarray}\label{eqpourcrosszero} \omega_{\mid T_{{\rm reg}}}\wedge \omega'_{\mid T_{{\rm reg}}}=0 \,\,{\rm in}\,\,H^0(T_{\rm reg},\Omega_{T_{\rm reg}}^{4}),
\end{eqnarray}
where $$\omega:={\rm pr}_2^*\sigma_X+\frac{-1+i\sqrt{3}}{2}{\rm pr}_3^*\sigma_X,\,
\omega':={\rm pr}_2^*\sigma_X-\frac{-1-i\sqrt{3}}{2}{\rm pr}_3^*\sigma_X.$$
The contradiction now comes from Lemma \ref{lemmeutilesurproduit} and
(\ref{eqpourcrossproduct}) which imply that either $\omega_{\mid T_{{\rm reg}}}=0$ or $\omega'_{\mid T_{{\rm reg}}}=0$, or
both  forms
$\omega_{\mid T_{{\rm reg}}}$ and $\omega'_{\mid T_{{\rm reg}}}$
 have rank $\leq 4$ at any point $t$  of $T$ and more precisely, at any point $t\in T_{{\rm reg}}$, are pulled-back
 via  a quotient map $T_{T,t}\rightarrow T'$, with  ${\rm dim}\,T'\leq4$.
 The form $\omega_{\mid T}$ cannot be $0$ because this would imply that
 the projection of $T$ in $X\times X$ via $({\rm pr}_2,\,{\rm pr}_3)$ is Lagrangian for a form which has rank $4n$
everywhere  on $X\times X$ while  by assumption ${\rm dim}\,({\rm pr}_2,\,{\rm pr}_3)(T)=3n$.
The same argument  also works for $\omega'$.
We thus conclude that the last possibility should hold. In that case,  the restrictions to $T$ of
  ${\rm pr}_2^*\sigma_X$ and ${\rm pr}_3^*\sigma_X$
 are also pulled-back via the quotient map $T_{T,t}\rightarrow T'$ hence have rank $\leq 4$.
The form $\sigma_X$ on $X$ is everywhere nondegenerate and
the projections ${\rm pr}_2,\,{\rm pr}_3$ restricted to $T$ are dominant, so we conclude that
the forms ${\rm pr}_2^*\sigma_X,\,{\rm pr}_3^*\sigma_X$ restricted to $T$  have rank equal to ${\rm dim}\,X$. As
 they  are  of rank $\leq 4$
 at a general point $t\in T$,  we get  a contradiction
if  $n\geq3$. If $n=2$, these  forms pulled-back to $T$ have rank $4$ but
they do not have the same kernel, because their respective kernels at a point $t\in T_{{\rm reg}}$  are
the spaces ${\rm Ker}\,({\rm pr}_{2\mid T})_* $, ${\rm Ker}\,({\rm pr}_{3\mid T})_*$ which are
 different generically on $T$ by the assumption that $({\rm pr}_2,{\rm pr}_3)$ is generically of maximal rank.
This contradiction concludes the proof.
\end{proof}
We construct  in the next subsections  triangle varieties for  the main `` known''  classes  of hyper-K\"ahler manifolds, for which we have
an explicit projective model.
\subsection{Hilbert schemes of $K3$ surfaces \label{sechilbtriangle}}
Recall from
\cite{beauvoi} that a projective $K3$ surface $S$ has a canonical $0$-cycle $o_S$ of degree $1$
with the following property : for any integer $k\geq 1$, the  degree-$k$ $0$-cycle
$ko_S$  on $S$ has a $k$-dimensional orbit
$$O_{ko_S}=\{z\in S^{(k)},\,z=ko_S\,\,{\rm in}\,\,{\rm CH}_0(S)\}$$
 in $S^{(k)}$ for rational equivalence
on $S$. An explicit example of   a
$k$-dimensional  orbit component of $ko_S$ and of  a triangle variety for $S^{[n]}$ is as follows.
Assume $S$ has a very ample polarization $L\in{\rm Pic}\,S$ with ${\rm deg}\,L^2=2g-2$. Let $k=2g-2$.
One component of the  orbit $O_{L^2}\subset S^{(2g-2)}$ of the zero-cycle $L^2$
   is birational to the Grassmannian $G(2,H^0(S,L))$
and is made of complete intersections $H_1\cap H_2$,
with $H_1,\, H_2\in|L|$, or rather of their supports.
Note that ${\rm dim}\,G(2,H^0(S,L))=2g-2$ as we want.
Assume furthermore that $2g-2=3n$ is divisible by $3$ and consider
\begin{eqnarray} T:=\{(z_1,\,z_2,\,z_3)\in (S^{[n]})^3,\,\,c(z_1)+c(z_2)+c(z_3)\in O_{L^2}\subset
 S^{(3n)}\},
\end{eqnarray}
where $c:S^{[l]}\rightarrow S^{(l)}$ denotes the Hilbert-Chow morphism.
\begin{prop} $T$ is a triangle variety for $S^{[n]}$.
\end{prop}
\begin{proof} The relation $({\rm pr}_1^*\sigma_{S^{[n]}}+{\rm pr}_2^*\sigma_{S^{[n]}}+{\rm pr}_3^*\sigma_{S^{[n]}})_{\mid T_{{\rm reg}}}$ follows  from the fact that the $0$-cycle
$c(z_1)+c(z_2)+c(z_3)$ is constant in ${\rm CH}_0(S)$ along $T$ and from  Mumford's theorem \cite{mumford}  because  the holomorphic $2$-form $\sigma_{S^{[n]}}$ is induced by the holomorphic $2$-form $\sigma_S$ via the incidence correspondence.
The fact that the dimension of $T$ is $3n$ follows from the fact that  $T$ is birational
 to  a generically finite cover of $O_{L^2}$ which has dimension $2g-2=3n$.
It remains to see that $T$ dominates the three summands and that it maps in a
generically finite to one way to its images in the three products $S^{[n]}\times S^{[n]}$.
The first statement  follows from the fact that $L$ is very ample
with $h^0(S,L)=g+1$, where $3n=2g-2$. This implies that for a general set
$z_1=\{x_1,\ldots,\,x_n\}$ of $n$ points
 of $S$, there is a reduced complete intersection $Z$ of two members of $|L|$ containing all the $x_i$.
 Then writing $Z$ as the union $z_1\sqcup z_2\sqcup z_3$ of three sets of cardinality $n$, we have
 $(z_1,z_2,z_3)\in T$.

For the second statement, we observe  that for such a general  reduced
$0$-dimensional complete intersection
$$Z=H_1\cap H_2=\{x_1,\ldots,x_{2g-2}\},$$
with $2g-2=3n$,
the first $2n$ points $x_1,\ldots,\,x_{2n}$ already impose $g-1$ conditions to $|L|$,
hence the space of hypersurfaces in $|L|$ containing
 these $2n$ points is the projective line $\langle H_1,\,H_2\rangle$.
Setting
$$z_1=\{x_1,\ldots,x_{n}\},\,z_2=\{x_{n+1},\ldots,x_{2n}\},\,\,z_3=\{x_{2n+1},\ldots,x_{3n}\},$$
we have $(z_1,z_2,z_3)\in T$ and the
fiber of the projection $p_{12}: T\rightarrow S^{[n]}\times S^{[n]}$ over $(z_1,z_2)$ consists by definition
of the single element  $z_3$.
\end{proof}
The numerical  condition $3n=2g-2$ used above is not important, as there are variants of this construction, starting from other Lagrangian subvarieties of $S^{[3n]}$, also obtained
as components of dimension $3n$ of the orbit of $3no_S$ in $S^{(3n)}$.
\subsection{Fano variety of lines in a cubic fourfold\label{subsecvarofline}}
The Fano variety $F_1(Y)$  of lines in a smooth cubic fourfold $Y$ is a hyper-K\"ahler
fourfold (see \cite{beaudo}). In this case, the triangles are just triangles in a usual sense,
namely the plane sections of $Y$ which are the unions of three lines
(plus an ordering of these lines). They form
a $6$-dimensional subvariety of $F_1(Y)^3$. Indeed, for each line $l\subset Y$, consider
the $\mathbb{P}^3_l$ of
planes containing $l$. Each of these planes  cuts $Y$ along the union of $l$ and a conic,
 and when the conic is degenerate, that is along a surface in $\mathbb{P}^3_l$,
 the conic becomes the  union of two lines, which together with
  $l$ form a triangle.
In this case, the fact that the family
of these triangles is   a Lagrangian subvariety of $F_1(Y)^3$
 is a consequence of Mumford's theorem \cite{mumford}. Indeed
 we know  that, via the incidence correspondence
 $p:P\rightarrow F_1(Y),\,q:P\rightarrow Y$
 given by the universal family of lines in $Y$,
 one has $\sigma_{F_1(Y)}=P^*\eta_Y$ for some class  $\eta_Y\in H^1(Y,\Omega_Y^3)$, and furthermore,
 for any triangle $([l],[l'],[l''])\in F_1(Y)^3$
   $$P_*l+P_*l'+P_*l''=h^3\,\,{\rm in}\,\,{\rm CH}_1(Y),$$ where $h=c_1(\mathcal{O}_Y(1))$.
\subsection{Debarre-Voisin hyper-K\"ahler fourfolds \label{secbazhov}}
Let $V_{10}$ be a $10$-dimensional vector space and let
$\lambda\in\bigwedge^3V_{10}^*$. The associated Debarre-Voisin fourfold
$F_\lambda\subset G(6,V_{10})$ is the set of $6$-dimensional vector subspaces $W_6\subset V_{10}$ such that
$\lambda_{\mid W_6}=0$. This is a hyper-K\"ahler fourfold for a general parameter $\lambda$
 (see \cite{debarrevoisin}).
 In \cite{bazhov},  Bazhov considered the  subvariety
$T\subset F_\lambda\times F_\lambda\times F_\lambda$
parameterizing triples $([W],[W'],[W''])\in F_\lambda^3$ such that the three subspaces  $ W,W'$ and $W''$
of $V_{10}$ generate only a $V_9\subset V_{10}$. He proved the following:
\begin{theo}(i) $ T $ has dimension $6$. It is birationally equivalent via the projection
 $p_{12}$ to the incidence subvariety $I\subset F_\lambda\times F_\lambda$
 defined as the set of couples $([W],[W'])$ such that $ W$ and $W'$  generate
 only a $V_9\subset V_{10}$.

 (ii) $I$ dominates $F_\lambda$ by the first projection.

(iii) $T$ is  a Lagrangian subvariety of $F_\lambda\times F_\lambda \times F_\lambda$.
\end{theo}
These three facts together  say that $T$ is  a triangle variety for $F_\lambda$.
\subsection{Double EPW sextics}
The double EPW sextics $X$ constructed by O'Grady in  \cite{ogrady} are quasi-\'etale double covers of sextic hypersurfaces
in $\mathbb{P}^5$ singular along a surface discovered by Eisenbud-Pospescu-Walter \cite{epw}.
We will follow here the description given by Iliev and Manivel
\cite{maniiliev}, which is very convenient to study subvarieties and relations between zero-cycles of  $X$. More precisely, the
Iliev-Manivel description provides a Fano fourfold  $Y$, such that $X$ parameterizes $1$-cycles in $Y$ and the $(2,0)$-form
on $X$ is induced via the incidence relation from a
 cohomology class of Hodge type $(3,1)$  in $Y$.
 By Proposition \ref{pronexistepasdanschow}, we cannot obtain enough relations
(\ref{eqtriangletexte})  in ${\rm CH}_0(X)$ to construct using Mumford's theorem
triangle varieties in $X$, that is, Lagrangian subvarieties of $X^3$.
 In the present case, and this was also exploited in the
 case of the Fano variety of lines of  a cubic fourfold, it suffices to exhibit relations between the corresponding
$1$-cycles in ${\rm CH}_1(Y)$.

The Iliev-Manivel construction is as follows. Let $V_5$ be a
$5$-dimensional vector space and let
$\mathbb{G}:=G(2,5)\subset\mathbb{P}^9$. Let  $Y\subset \mathbb{G}$  be the generic  complete intersection of
a  linear Pl\"ucker  section $H\subset \mathbb{G}$ and a quadratic Pl\"ucker  section $Q$  of  $\mathbb{G}$.
The fourfold $Y$ is Fano of index $2$ with Picard number $1$ and its variety of conics
$\mathcal{H}_{2,0}$  is $5$-dimensional. It is fibered into
$\mathbb{P}^1$'s, because if $C\subset \mathbb{G}$ is a conic, there exists a hyperplane $V_4\subset V_5$
such that $C\subset G(2,V_4)$. (Indeed, the surface in $\mathbb{P}(V_5)$ swept-out by lines parameterized by $C$ has degree
$2$,  hence is contained in some $\mathbb{P}(V_4)\subset \mathbb{P}(V_5)$.) Thus $C$ is contained in
the del Pezzo surface $\Sigma=H\cap Q\cap G(2,V_4)$ which has index $1$ and degree $4$.
But then $C$ moves in a pencil in $\Sigma$.
Next Iliev and Manivel show that
$Y$ has a $(3,1)$-form $\eta_Y\in H^{3,1}(Y)$ and
considering the incidence diagram given by
the universal conic

 \begin{eqnarray}\label{numerodiagpourincidenceilma}
  \xymatrix{
&\mathcal{C}\ar[r]^{q}\ar[d]^{p}&Y\\
&\mathcal{H}_{0,2}&},
\end{eqnarray}
they  show that the $(2,0)$-form $p_*q^*\eta$ has generic rank $4$ on $\mathcal{H}_{0,2}$. It follows that
the base of the  MRC fibration of $\mathcal{H}_{0,2}$ is $4$-dimensional, with fibers given by the
$\mathbb{P}^1$'s described above. Finally, it is shown in \cite{maniiliev}  that this base is birational to
a general double EPW sextic $X$.

This construction   is very convenient  to exhibit
Lagrangian subvarieties
in $X^l$ and, for $l=1$, this is already done in \cite{maniiliev}. For example, the variety of conics
contained in a general hyperplane section $Y'\subset Y$ is $3$-dimensional and its image
 in $X$ is a Lagrangian surface
constructed in \cite{maniiliev}. This follows from the fact that the class $\eta$ vanishes on $Y'$ and
that the pull-back of $\sigma_X$ to $\mathcal{H}_{0,2}$ is defined as
$p_*(q^*\eta)$. We now explain how to use this description of $X$ to produce a triangle variety for
$X$.

First of all, we observe that nondegenerate  rational curves of degree $4$ on $Y$
are parameterized by
a $9$-dimensional variety $\mathcal{H}_{0,4}$, while nondegenerate  elliptic curves of degree
$6$ are parameterized by a $12$-dimensional variety $\mathcal{H}_{1,6}$. Furthermore, there is
a dominant rational map
\begin{eqnarray}\label{eqdomPhi}
\Phi: \mathcal{H}_{1,6}\dashrightarrow  \mathcal{H}_{0,4}
\end{eqnarray}
with general fiber $\mathbb{P}^3$. This map is obtained by liaison. Indeed, a
nondegenerate rational curve
$C$  of degree $4$ on $Y$ has
$h^0(C,\mathcal{O}_C(1))=5$ and the restriction map
$H^0(Y,\mathcal{O}_Y(1))\rightarrow H^0(C,\mathcal{O}_C(1))$
is surjective, hence has a $4$-dimensional kernel.
As $C$ is general, $C$ is defined in $Y$ by linear Pl\"ucker equations.
Thus, taking three general equations $\sigma_1,\sigma_2,\sigma_3$ vanishing on $C$, the
locus  defined by these $3$ equations is a curve of degree $10$ that contains $C$ and
is the union of
$C$ and an elliptic curve of degree $6$. Conversely,
starting
from a nondegenerate elliptic curve $E$ of degree $6$, we have $h^0(E,\mathcal{O}_E(1))=6$, and
 the restriction map
$H^0(Y,\mathcal{O}_Y(1))\rightarrow H^0(E,\mathcal{O}_E(1))$
is surjective, hence has a $3$-dimensional kernel. The locus defined by this $3$-dimensional set of
linear  Pl\"ucker equations is a curve of degree $10$ containing $E$ and is in fact the union of $E$
and a residual rational curve of degree $4$.

There is a $4$-dimensional (or codimension $1$)
family $\Gamma_4\subset \mathcal{H}_{0,2}$ of conics in $Y$ (which must be contracted
to a surface in $X$),
which is constructed as follows.
Consider  the variety $Z:=H\cap \mathbb{G}$ and its variety of planes
$P\subset Z$. The equation defining $H$
is a $2$-form $\omega\in \bigwedge^2V_5^*$. It is well-known that a  plane
in $\mathbb{G}$ corresponds to a point
$x\in\mathbb{P}(V_5)$ together with a
$\mathbb{P}(V_4)\subset\mathbb{P}(V_5)$ passing through  $x$ and  defining the plane
$P$ of lines in $\mathbb{P}(V_4)$ passing through $x$. This plane is contained in $Z$ when
$V_4$ is contained in $x^{\perp\omega}$, which provides the desired $4$-dimensional family
(parameterized birationally by the choice of $x\in \mathbb{P}(V_5)$).
Any such a plane $P$ determines a conic
$C=P\cap Y$ in $Y$ (or is contained in $Y$, but this does not happen for generic $Y$). This provides us with a rational $4$-dimensional subvariety
$$ \Gamma_4\subset \mathcal{H}_{0,2}.$$
It is obvious that the subvariety of $X$ we get this way is Lagrangian for $\sigma_X$, because
it is dominated by the rational
variety $\Gamma_4$.

We now make the following construction. Inside $\Gamma_4\times \Gamma_4$, there is a
$6$-dimensional subvariety $\Gamma_6$ consisting of pairs of intersecting  conics.
We observe that $\Gamma_6$ maps naturally to $\mathcal{H}_{0,4}$, via the $2$ to $1$ map which
associates to a pair of intersecting conics  the rational curve of degree $4$
 which is the union of the two conics.
This way we get a $6$-dimensional variety parameterizing  degree $4$ rational curves in $Y$, and applying
the residual construction explained previously, we get a
$9$-dimensional subvariety $\Gamma_9^1$ of $\mathcal{H}_{1,6}$.

Let now $\mathcal{T}\subset \mathcal{H}_{0,2}\times \mathcal{H}_{0,2}\times \mathcal{H}_{0,2}$ be the
set of triples of conics $(C_1,\,C_2,\,C_3)$ in $Y$, intersecting
each other (a triangle of conics), and such that the
singular elliptic curve $E=C_1\cup C_2\cup C_3$ is a member of the family parameterized by
$\Gamma_9^1$.
\begin{theo}\label{theotriangleepw} For general $Y$, the image $T$ of $\mathcal{T}$ in $X^3$ is a triangle variety.
\end{theo}
\begin{proof}  The triples $(C_1,\,C_2,\,C_3)$ of conics in $Y$ parameterized by $\mathcal{T}$ have the property
that the singular elliptic curve
$E=C_1\cup C_2\cup C_3\subset Y$ is residual in $Y$ to a rational  curve of degree $4$ which is the union of two
conics $C_4,\,C_5$ meeting at one point, where $C_4$ and $C_5$ are cut on $Y$ by planes in $Z$. All the planes contained in  $Z$ are rationally
equivalent in $Z$, so we conclude that the elliptic curves $E$ parameterized by $\Gamma_9^1$  are all rationally equivalent. By
\cite{maniiliev}, the holomorphic $2$-form on $X$ pulls-back to a holomorphic $2$-form
$\widetilde{\sigma}_X$ on
$\mathcal{H}_{0,2}$ which is induced from a cohomology class of type $(3,1)$ on $Y$ via the incidence correspondence.
Mumford's theorem \cite{mumford} implies that ${\rm pr}_1^*\widetilde{\sigma}_X+{\rm pr}_2^*\widetilde{\sigma}_X+{\rm pr}_3^*\widetilde{\sigma}_X$ vanishes on $\mathcal{T}$, hence equivalently that
${\rm pr}_1^*{\sigma}_X+{\rm pr}_2^*{\sigma}_X+{\rm pr}_3^*{\sigma}_X$ vanishes on ${T}$.
We leave to the reader checking   the dimension count for general $Y$ and the fact that $T$ dominates
$X$ by the various projections and is generically finite on its image in $X\times X$ by the various projections.
\end{proof}
\begin{rema}{\rm The method described in the next section and the existence of a covering of $X$
by a family of  Lagrangian surfaces given in \cite{maniiliev}  can also be used to construct triangle varieties for $X$, see
Theorem \ref{theotrianglecoverlag}.}
\end{rema}
\subsection{Lagrangian fibrations\label{secjacfib} and Lagrangian coverings}
Let $\phi:X\rightarrow B$ be a projective  Lagrangian fibration on a hyper-K\"ahler manifold of dimension $2n$. Recall from
Lin's paper \cite{lin} that $\phi$ has a Lagrangian constant cycle  multisection $\widetilde{B}$.  By base change from $B$ to $\widetilde{B}$, we get (possibly after desingularization)  an induced fibration $\widetilde{X}\rightarrow \widetilde{B}$ which has a  section, hence is (over a dense open set of $\widetilde{B}$) a family of abelian varieties.
Let $\widetilde{I}:=\widetilde{X}\times_{\widetilde{B}}\widetilde{X}\subset \widetilde{X}\times \widetilde{X}$. Using the relative addition map, we get a rational
map $\mu:\widetilde{I}\dashrightarrow \widetilde{X}$ and finally we define $T$ as the image of
$\widetilde{I}$ in $X\times X\times X$ under the rational  map  $(r\circ {\rm pr}_1,r\circ {\rm pr}_2,r\circ -\mu)$ where $r:\widetilde{X}\rightarrow X$
is the natural map and the ${\rm pr}_i$'s are the projections from $\widetilde{X}\times \widetilde{X}$ to $ \widetilde{X}$, restricted to $\widetilde{I}$.
\begin{prop}\label{triangalelagrangien} The variety $T$ is a triangle variety.
\end{prop}

\begin{proof} As $T$ is the union over $t\in \widetilde{B}$ of the graphs of the sum map in the fibers $X_b$, it is clear that $T$ dominates $X$ by the three projections and
 maps in a generically finite way to the products  $X\times X$ of any two factors (the image is $X\times_BX$ but the map
 is not birational because of the base change $\widetilde{B}\rightarrow B$). We want to prove that ${\rm pr}_1^*\sigma_X+{\rm pr}_2^*\sigma_X+{\rm pr}_3^*\sigma_X=0$ on $T_{\rm reg}$, or, equivalently \begin{eqnarray}\label{eqamontrer} {\rm pr}_1^*\sigma_{\widetilde{X}}+{\rm pr}_2^*\sigma_{\widetilde{X}}=\mu^*\sigma_{\widetilde{X}}\end{eqnarray}
on $\widetilde{I}$, where
$\sigma_{\widetilde{X}}:=r^*\sigma_X$. As $\tilde{\phi}:\widetilde{X}_{\rm reg}\rightarrow \widetilde{B}_{\rm reg}$ is a Lagrangian fibration  with respect to $\sigma_{\widetilde{X}}$,
$\sigma_{\widetilde{X}|\widetilde{X}_{\rm reg}}\in H^0(\widetilde{X}_{\rm reg},F^1\Omega_{\widetilde{X}_{\rm reg}}^2)$, where $ F^1\Omega_{\widetilde{X}^2_{\rm reg}}:= \tilde{\phi}^*\Omega_{\widetilde{B}}\wedge\Omega_{\widetilde{X}_{\rm reg}}$. Let
$F^2\Omega^2_{\widetilde{X}_{\rm reg}}=\tilde{\phi}^*\Omega_{\widetilde{B}}^2$. The quotient bundle $F^1\Omega_{\widetilde{X}^2_{\rm reg}}/F^2\Omega_{\widetilde{X}^2_{\rm reg}}$ is isomorphic to $\tilde{\phi}^*\Omega_{\widetilde{B}}\otimes\Omega_{\widetilde{X}_{\rm reg}/\widetilde{B}}$. We have \begin{eqnarray}\label{eqamontrerpresque}{\rm pr}_1^*\sigma_{\widetilde{X}}+{\rm pr}_2^*\sigma_{\widetilde{X}}=\mu^*\sigma_{\widetilde{X}}
\end{eqnarray}
in $H^0(\widetilde{I}_{\rm reg},F^1\Omega_{\widetilde{I}^2_{\rm reg}}/F^2\Omega_{\widetilde{I}^2_{\rm reg}})$ because on the fibers $\widetilde{X}_b$, we have
$$\mu^*\alpha={\rm pr}_1^*\alpha+{\rm pr}_2^*\alpha$$
for any $\alpha\in H^0(\widetilde{X_b},\Omega_{\widetilde{X_b}})$,
so that
$$\mu^*={\rm pr}_1^*+{\rm pr}_2^*: \mu^*\Omega_{\widetilde{X}_{\rm reg}/\widetilde{B}}\rightarrow
 \Omega_{\widetilde{I}_{\rm reg}/\widetilde{B}}. $$
It follows from (\ref{eqamontrerpresque}) that
${\rm pr}_1^*\sigma_{\widetilde{X}}+{\rm pr}_2^*\sigma_{\widetilde{X}}-\mu^*\sigma_{\widetilde{X}}\in H^0(\widetilde{I}_{\rm reg},\tilde{\phi}^*\Omega_{\widetilde{B}_{\rm reg}}^2)\subset H^{2,0}(\widetilde{I}_{\rm reg})$, which gives an equality of $2$-forms on
$\widetilde{I}_{\rm reg}$
\begin{eqnarray}\label{eqamontrerpresquepresque}{\rm pr}_1^*\sigma_{\widetilde{X}}+{\rm pr}_2^*\sigma_{\widetilde{X}}-\mu^*
\sigma_{\widetilde{X}}
=\tilde{\phi}^*\eta
\end{eqnarray}
for some $\eta\in H^0(\widetilde{B},\Omega^2_{\widetilde{B}_{\rm reg}})$.
On the other hand, recall that the multisection
$\widetilde{B}$ of $\phi$, or $0$-section
$\widetilde{B}$ of $\tilde{\phi}$, was chosen to be Lagrangian
for $\sigma_X$ (or equivalently $ \sigma_{\widetilde{X}}$). Restricting (\ref{eqamontrerpresquepresque}) to the $0$-section $\widetilde{B}$, we then conclude that $\eta=0$, which proves (\ref{eqamontrer}).
\end{proof}
Let us say that a hyper-K\"ahler manifold $X$ has a Lagrangian covering if
there exists a diagram
\begin{eqnarray}\label{numerodiagpourlacover}
  \xymatrix{
&\mathcal{L}\ar[r]^{\Phi}\ar[d]^{\pi}&X\\
&B&}
\end{eqnarray}
where $\mathcal{L}$ and $B$ are smooth projective varieties, the morphism
$\Phi$ is surjective and maps  birationally  the general
fiber ${L}_t,\,t\in B$, of $\pi$ to a (possibly singular) Lagrangian
subvariety of $X$, and furthermore, the following condition holds.
 As $\Phi(L_t)$ is Lagrangian, one has a
 natural morphism of coherent sheaves
 $$\lrcorner\sigma_X: N_{L_t/X}\rightarrow \Omega_{L_t}$$
 which is a generic isomorphism, and  induces a morphism at the level of global sections
 $$H^0(L_t, N_{L_t/X})\rightarrow H^0(L_t,\Omega_{L_t}).$$
 We ask that for $t\in B$ generic, the composite map
 $$T_{B,t}\rightarrow  H^0(L_t, N_{L_t/X})\stackrel{\lrcorner\sigma_X}{\rightarrow }H^0(L_t,\Omega_{L_t})$$
where the first map is the classifying map, is an isomorphism. In particular ${\rm dim}\,B=h^{1,0}({L}_t)=:g$.
This  condition is satisfied by unobstructedness results for deformations of Lagrangian
submanifolds (see \cite{voisinlag})  if,  for general $t\in B$, the fiber
${L}_t$ is isomorphic via $\Phi$ to a smooth Lagrangian subvariety of $X$, and $\mathcal{L}\rightarrow
B$
identifies near $t$ to the universal family
 of deformations of ${L}_t$ in $X$.
 For singular Lagrangian subvarieties, the deformation theory is
 not well understood.
Note that, with the hypotheses above, the surjectivity of $\Phi$ has the following interpretation.
\begin{lemm} \label{lealblag} The surjectivity of $\Phi$ is equivalent to the fact   that the Albanese
map ${\rm alb}_{L_t}:L_t\rightarrow {\rm Alb}\,L_t$ is generically finite on its image for general $t$.
\end{lemm}
\begin{proof} The second property is equivalent to the fact that, for general $t\in B$,
 ${\rm alb}_{L_t}$ has a generically injective differential, or equivalently, that the
 evaluation map
 \begin{eqnarray}\label{eavalpourforme} {\rm ev} : H^0(L_t,\Omega_{L_t})\otimes \mathcal{O}_{L_t}\rightarrow \Omega_{L_t}
 \end{eqnarray}
 is generically surjective on $L_t$.
 The surjectivity of $\Phi$ is equivalent to the fact  that $\Phi$  is submersive generically along
 $L_t$ for general $t$.
As $L_t$ imbeds generically into $X$,  this is well-known to be equivalent to the fact that
the   map
 \begin{eqnarray}\label{eavalpourfnormale} {\rm ev}: T_{B,t}\otimes \mathcal{O}_{L_t}\rightarrow
  N_{L_t/X},
 \end{eqnarray}
 which is the composition of the evaluation map and of  the classifying map $T_{B,t}\rightarrow H^0(L_t, N_{L_t/X})$
 is generically surjective.
 We now use the map
 $$\lrcorner\sigma_X: N_{L_t/X}\rightarrow \Omega_{L_t}$$
 which is a generic isomorphism on $L_t$, and  induces a morphism at the level of global sections which
 composed with the classifying map makes  the
 following diagram commutative
  \begin{eqnarray}\label{diagpournorfor}
  \xymatrix{
& T_{B,t}\otimes \mathcal{O}_{L_t}\ar[d]^{\lrcorner\sigma_X}\ar^{\,\,\,\,\,\,\,\,\,\,\,\,\,\,{\rm ev}}[r]&N_{L_t/X}\ar[d]^{\lrcorner\sigma_X}\\
&H^0(L_t,\Omega_{L_t})\otimes \mathcal{O}_{L_t}\ar^{\,\,\,\,\,\,\,\,\,\,\,\,\,\,{\rm ev}}[r]& \Omega_{L_t}},
\end{eqnarray}
  As the first vertical map is by assumption an isomorphism, it follows that
 the generic surjectivity of the evaluation map (\ref{eavalpourfnormale}) is equivalent to   the generic surjectivity of the evaluation map (\ref{eavalpourfnormale})
\end{proof}
We show the following variant of Proposition \ref{triangalelagrangien}.
\begin{theo}\label{theotrianglecoverlag} Let $X$ be a projective hyper-K\"ahler manifold admitting a
Lagrangian covering $\Phi:\mathcal{L}\rightarrow X$. Assume that the general fibers $L_t$
have the property that
the sum map $L_t\times L_t\rightarrow {\rm Alb}\,L_t$ is surjective (in particular
$2n\geq g$).
Assume there exists  a  Lagrangian subvariety  $K\subset X$  such that the general fiber
$\Phi({L}_t)$ intersects $K$  in  a finite (nonzero) number of   points. Then $X$ admits a triangle subvariety.
\end{theo}
Note that by the same arguments as above,  the assumption on $K$ will be satisfied by taking $K=L_s$, for general $s$,  assuming that
the fibers $L_t\subset X$ are  smooth Lagrangian, and a general
form $\alpha\in H^0(L_t,\Omega_{L_t})$ has finitely many zeroes.
\begin{proof}[Proof of Theorem \ref{theotrianglecoverlag}] Consider, over the open set $B_{\rm reg}$ of regular values of
$\pi$, the Albanese fibration
$\mathcal{A}\rightarrow B_{\rm reg}$ with fiber ${\rm Alb}\,L_b$ over $b\in B$.
By assumption,  a general   variety $\Phi(L_b)\subset X$ intersects $K$ in finitely
many points, which provides a generically finite cover
$$B_K=\Phi^{-1}(K)\rightarrow B_{\rm reg}$$
parameterizing the pairs $(b,k)$, where $k\in L_b$ is such that $\Phi(k)\in K$.
Denoting by $\mathcal{L}_K$ the fibered product
$\mathcal{L}\times_BB_K$, there is a natural section
$$\sigma:B_K\rightarrow \mathcal{L}_K=\mathcal{L}\times_BB_K,$$
$$(b,k)\mapsto k.$$
We denote by ${\rm alb}_K:\mathcal{L}_{K}\rightarrow \mathcal{A}_K=\mathcal{A}\times_{B_{\rm reg}}B_K$
 the relative Albanese
map defined by the section $\sigma$, so that
$${\rm alb}_K(x)={\rm alb}_{L_b}(x-\sigma(b)),$$
where $b=\pi(x)\in B_K$.
Let now, for any integer $N\not=0$,  $\widetilde{T}_N\subset \mathcal{L}_K\times_{B_K}\mathcal{L}_K\times_{B_K}\mathcal{L}_K$
be defined as
\begin{eqnarray}\label{eqdefTlag} \widetilde{T}:=\{(x,y,z)\in
\mathcal{L}_K\times_{B_K}\mathcal{L}_K\times_{B_K}\mathcal{L}_K,\,
N({\rm alb}_K(x)+{\rm alb}_K(y)+{\rm alb}_K(z))=0
\\
\nonumber
\,\,{\rm in}\,\,{\rm Alb}\,L_b,\,
b:=\pi(x)=\pi(y)=\pi(z)\}.
\end{eqnarray}
The variety $\mathcal{L}_K$ has a morphism
$\Phi_K:\mathcal{L}_K\rightarrow X$ composed from $\Phi$ and the natural map
$\mathcal{L}_K\rightarrow \mathcal{L}$.
We define $T_N$  as the Zariski closure in $X\times X\times X$ of
$(\Phi_K,\Phi_K,\Phi_K)(\widetilde{T}^0_N)$,  where $\widetilde{T}^0_N$ is the union of the  irreducible components
of $\widetilde{T}_N$
dominating $B_K$, and where the point
${\rm alb}_K(x)+{\rm alb}_K(y)+{\rm alb}_K(z) $ is of order exactly  $N$.
It remains to show that $T_N$  has the required properties for large $N$.
First of all, the   proof given for   Proposition \ref{triangalelagrangien} works as well to show
\begin{lemm} One has $({\rm pr}_1^*\sigma_X+{\rm pr}_2^*\sigma_X+{\rm pr}_3^*\sigma_X)_{\mid T_{N,{\rm reg}}}=0$.
Equivalently,
$({\rm pr}_1^*(\Phi_K^*\sigma_X)+{\rm pr}_2^*(\Phi_K^*\sigma_X)
+{\rm pr}_3^*(\Phi_K^*\sigma_X))_{\mid \widetilde{T}_{N,{\rm reg}}}=0$.
\end{lemm}
We next observe that, if $X$ is of dimension $2n$, $\widetilde{T}_N$ has  expected
 dimension $3n$, which is the dimension of  a triangle variety. Indeed, let
 $g:={\rm dim}\,B={\rm dim}\,{\rm Alb}\,L_b$. Then, as ${\rm dim}\,L_b=n$,
  $${\rm dim}\,\mathcal{L}_K\times_{B_K}\mathcal{L}_K\times_{B_K}\mathcal{L}_K=g+3n,$$
  while from  (\ref{eqdefTlag}), we see that
    $\widetilde{T}^0_N$ is the inverse image of the $N$-torsion multisection of $\mathcal{A}_K\rightarrow B_K$
    via the sum morphism ${\rm alb}_K\circ {\rm pr}_1+{\rm alb}_K\circ {\rm pr}_2+{\rm alb}_K\circ {\rm pr}_3$, over
    the regular locus $B_K^0$ of $\mathcal{L}_K\rightarrow B_K$.
    Hence the expected codimension of
    $\widetilde{T}_N$ is $g$ and the expected dimension of $\widetilde{T}_N$    is $3n$.
The proof of the theorem concludes with
\begin{lemm} Under the assumptions of the theorem,
$\widetilde{T}^0_N$ is actually of dimension $3n$,  the projections
$\widetilde{T}^0_N\rightarrow \mathcal{L}_K$ are dominant and  the projections
$\widetilde{T}^0_N\rightarrow \mathcal{L}_K\times_{B_K}\mathcal{L}_K$ are generically finite on their images.
\end{lemm}
\begin{proof} By assumption,  the sum map ${L}_t\times { L}_t\rightarrow {\rm Alb}\,L_t$ is surjective for general $t$,
while by Lemma \ref{lealblag}, the Albanese map $L_t\rightarrow {\rm Alb}\,L_t$ is generically finite
on its image.  (Here the Albanese map of $L_t$ is computed using one of the finitely many
points of $L_t\cap K$, in other words, $t$ is  taken in $B_K$ rather than $B$.) This implies that for general  $x\in  L_t$, there is a
solution to the equation
\begin{eqnarray}\label{eqtroispoints}N( {\rm alb}_{L_t}\,x + {\rm alb}_{L_t}\,y+{\rm alb}_{L_t}\,z)=0,
\end{eqnarray}
with $y$, $z\in L_t$.
This is saying that the three projections  $\widetilde{T}\rightarrow \mathcal{L}_K$  are surjective.
Finally, using (\ref{eqtroispoints}),
we find that the projections $\widetilde{T}_N\rightarrow  \mathcal{L}_K\times_{B_K}\mathcal{L}_K$ are generically finite on their image
because the Albanese map of $L_t$ is generically finite on its image.
\end{proof}
It remains to see that the same properties hold for
$T_N\subset X\times X\times X\times X$.
This follows from the following
lemma which is proved exactly as
Lemma \ref{lealblag}.
\begin{lemm}\label{lealblag2} The assumptions that
the sum map $L_t\times L_t\rightarrow {\rm Alb}\,L_t$ is surjective is equivalent to the fact that
the natural map
$(\Phi,\Phi):\mathcal{L}_{K}\times_{B_K} \mathcal{L}_{K}\rightarrow X\times X$ is generically finite on its image.
\end{lemm}
As $\Phi$ is surjective, the fact that
the projections  ${\rm pr}_i:\widetilde{T}^0_N\rightarrow \mathcal{L}_K$ are dominant for $i=1,\,2,\,3$  implies the same property
for the projections  ${\rm pr}_i:{T}_N\rightarrow X$.
As $(\Phi,\Phi):\mathcal{L}_{K}\times_{B_K} \mathcal{L}_{K}\rightarrow X\times X$ is generically finite on its image,
the fact that
the projections  ${\rm pr}_{ij}:\widetilde{T}^0_N\rightarrow \mathcal{L}_K\times_{B_K} \mathcal{L}_K$ are dominant for $i=1,\,2,\,3$
does not necessarily implies   the same property
for the projections  ${\rm pr}_{ij}:{T}_N\rightarrow X\times X$, but it will imply it
if  $N$ is  large, using the Zariski density of torsion points.
\end{proof}
\begin{example} In the case of the variety of lines $X=F_1(Y)$ of a smooth cubic fourfold, we get
by applying Theorem \ref{theotrianglecoverlag}
constructions of  triangle varieties for $X$, different from the one constructed in Section
\ref{subsecvarofline},  by using its Lagrangian covering by Fano surfaces
$S_H:=F_1(Y_H)$, where $Y_H\subset Y$ is a hyperplane section $Y\cap H$ of $H$, and $S_H$ is
its surface of
 lines. The construction depends on the choice of a Lagrangian surface $K\subset X$.
\end{example}
\section{Construction of surface decompositions from triangle varieties \label{secconst}}
Let $X$ be a smooth projective variety of dimension $2n$ and
$\sigma_X\in H^{2,0}(X)$ a holomorphic $2$-form on
$X$. First of all, note that from a triangle variety $T\subset X\times X\times X$, we can construct
for each $k\geq 3$ a subvariety $T_k$ of $X^k$ of dimension $kn$ satisfying the following property: the holomorphic $2$-form
$\sum_{i=1}^k\epsilon_i{\rm pr}_i^*\sigma_X$ vanishes on $T_k$, with $\epsilon_i=\pm 1$.
The $k$-angle variety
$T_k$ is defined inductively  by composition in the sense of correspondences. For $k=4$, let
\begin{eqnarray}\label{eqTk} T'_4={\rm pr}_{1245*}({\rm pr}_{123}^{-1}(T) \cap {\rm pr}_{345}^{-1}(T))\subset X^4,
\end{eqnarray}
where the projections are defined on $X^5$, ${\rm pr}_{1245}$ takes value in $X^4$ and
${\rm pr}_{123},\, {\rm pr}_{345}$ take value in $X^3$.
On ${\rm pr}_{123}^{-1}(T)$, one has ${\rm pr}_1^*\sigma_X+{\rm pr}_2^*\sigma_X+{\rm pr}_3^*\sigma_X=0$ and on
${\rm pr}_{345}^{-1}(T)$, one has ${\rm pr}_3^*\sigma_X+{\rm pr}_4^*\sigma_X+{\rm pr}_5^*\sigma_X=0$ so that, by subtracting,  one has on
the regular locus of $ {\rm pr}_{123}^{-1}(T) \cap {\rm pr}_{345}^{-1}(T)$, hence also on $T'_{4,reg}$ :
  \begin{eqnarray}
 \label{eqformequisannu}{\rm pr}_1^*\sigma_X+{\rm pr}_2^*\sigma_X-{\rm pr}_4^*\sigma_X-{\rm pr}_5^*\sigma_X=0,
 \end{eqnarray}
   where now the projections are defined on $X^4$ with factors indexed by $1,2,4,5$.
As $T$ dominates $X$ by the projections, the variety $T'_4$ so defined also dominates
$X$ by the various projections. As the fibers of the projection $T\rightarrow X$ have dimension
$n$, $T'_4$ has at least one component which is of dimension $\geq 4n$.   We take for $T_4$
the union of the irreducible components of dimension $4n$ of $T_4$.
Note that, if $X$ is hyper-K\"ahler, the $2$-form $\sigma_X$ is everywhere nondegenerate, so
$T_4$ does not have components of  dimension $>4n$, because we already know
 by
(\ref{eqformequisannu}) that the components are Lagrangian for the holomorphic symplectic form ${\rm pr}_1^*\sigma_X+{\rm pr}_2^*\sigma_X-{\rm pr}_4^*\sigma_X-{\rm pr}_5^*\sigma_X$ on $X^4$.
The variety $T_k$ is similarly defined inductively by composing
$T_{k-1}$ and $T$.

 Recall that for $X$ as above, an algebraically coisotropic subvariety $Z\subset X$ of dimension $n+1$ admits a rational map $\tau:Z\dashrightarrow \Sigma$, where    $\Sigma$ is a surface and
 $\sigma_{X\mid Z_{\rm reg}}=\tau_{\rm reg}^*\sigma_\Sigma$ for some holomorphic $2$-form
 $\sigma_\Sigma$ on $\Sigma$.
\begin{theo}\label{proexistesurfdec} Let $X$ be a  projective hyper-K\"ahler variety of dimension $2n$. Assume $X$ has a triangle variety
$T\subset X^3$ and an  algebraically coisotropic subvariety  $\tau:Z\dashrightarrow \Sigma$ of dimension $n+1$. Let $F\subset X$ be the general fiber of $\tau$.
 Then if the intersection of $F^n\subset X^n$ and  ${\rm pr}_{1\ldots n}(T_{n+1})\subset X^n$ satisfies
 \begin{eqnarray}\label{eqintnonnul}F^{n}\cdot {\rm pr}_{1\ldots n}(T_{n+1})\not=0,
 \end{eqnarray}
$X$ admits a surface decomposition. In particular it admits  mobile algebraically coisotropic subvarieties of any codimension
$\leq n$.
\end{theo}
In (\ref{eqintnonnul}), we have ${\rm dim}\,F=n-1$, so ${\rm dim}\,F^{n}=n(n-1)$ and ${\rm dim}\,T_{n+1}=n(n+1)$, while the intersection takes place in $X^{n}$ which has dimension $2n^2=n(n-1)+n(n+1)$.
\begin{proof} We construct $\phi:\Gamma_0\rightarrow X,\,\psi: \Gamma_0\rightarrow \Sigma^n$ by the formulas
\begin{eqnarray}\label{eqGamma0} \Gamma_0={\rm pr}_{1\ldots n}^{-1}(Z^n)\cap T_{n+1}\subset T_{n+1}\subset X^{n+1},
\\
\nonumber
\phi: ={\rm pr}_{n+1}:\Gamma_0\rightarrow X,\,\,\psi=\tau^n\circ {\rm pr}_{1\ldots n}:\Gamma_0\rightarrow \Sigma^n.
\end{eqnarray}
As $\Gamma_0\subset T_{n+1}$, the form $\sum_{i=1}^{n+1}\epsilon_i{\rm pr}_i^*\sigma_X$ vanishes on $\Gamma_{0,reg}$, where the $\epsilon_i$ are the signs introduced above. In other words, using $\phi={\rm pr}_{n+1}$
\begin{eqnarray}\phi^*\sigma_X=\sum_{i=1}^n\epsilon'_i{\rm pr}_i^*\sigma_{X\mid \Gamma_{0,reg}}
\end{eqnarray}
in $H^0(\Omega^2_{\Gamma_{0,reg }})$, where $\epsilon'_i=\pm \epsilon_i$. We next use the fact that ${\rm pr}_i(\Gamma)\subset Z$ and  that
$\sigma_{X\mid Z}=\tau^*(\sigma_\Sigma)$. We then get the desired formula characterizing
 a surface decomposition
\begin{eqnarray}\label{egaquivaservir} \phi^*\sigma_X=\psi^*(\sum_{i=1}^n\epsilon'_i {\rm pr}_i^*(\sigma_\Sigma))\,\,{\rm in}\,\,H^{2,0}(\Gamma_{0,reg}).
\end{eqnarray}
We need to show that $\phi$ and $\psi$ are dominant, and that we can assume that they are generically finite.
The fact that $\psi$ is dominant is a consequence of (\ref{eqintnonnul}), which can be seen as saying that ${\rm pr}_{1\ldots n}^{-1}(Z^n)\cap T_{n+1}$ intersects nontrivially the fibers of
$\tau^n:{\rm pr}_{1\ldots n}^{-1}(Z^n)\rightarrow \Sigma^n$.
Knowing that $\psi$ is dominant, we conclude that the form $\psi^*(\sum_{i=1}^n\epsilon'_i {\rm pr}_i^*(\sigma_\Sigma))$ has generic rank $2n$ on $\Gamma_{0,reg}$. It thus follows
 from (\ref{egaquivaservir}) that $\phi^*\sigma_X$
 has  generic rank $2n$ on $\Gamma$, hence that $ \phi$ is also dominant. The last argument applies to any irreducible component $\Gamma'_0 $ of $\Gamma_0$ dominating $\Sigma^n$, which thus also has to dominate $X$. Finally, by cutting $\Gamma'_0$ by hyperplane sections and reapplying the same arguments  if necessary, we get a $\Gamma$ which  is generically finite onto both $\Sigma^n$ and $X$, and still
 satisfies (\ref{egaquivaservir}).
\end{proof}
We  also have the following result, whose proof is a variant of that of  Theorem \ref{proexistesurfdec},  and shows how to construct new algebraically coisotropic subvarieties out of old ones, using a triangle variety:

 Let $X$ be smooth projective variety of dimension $2n$  with an everywhere nondegenerate holomorphic $2$-form $\sigma_X$.
Denote by  $I_{r}\subset X$  an algebraically coisotropic subvariety of $X$ of codimension
$n-r$. Hence there exists  a rational map
$$\phi_r:I_r\dashrightarrow B_r$$
to a smooth projective variety $B_r$ of dimension $2r$, with general  fiber $F_r$ of dimension $n-r$, such that
\begin{eqnarray}\label{eqalgcoisio}
\sigma_{X\mid I_r}=\phi_r^*\sigma_{B_r},
\end{eqnarray}
for some holomorphic $2$-form $\sigma_{B_r}$ on $B_r$ which is  generically of maximal rank $2r={\rm dim}\,B_r$.
\begin{theo}\label{theovariant} Assume that $X$ has a triangle variety $T$ relative to $\sigma_X$ and let
$I_r,\,I_{r'}$ be two algebraically coisotropic varieties of $X$ of respective codimensions
$r,\,r'$. Assume that

 (*) the class ${\rm pr}_{3*}([T]\cup {\rm pr}_{12}^*([I_r\times I_{r'}]))$ is nonzero in $H^{2n-2r-2r'}(X,\mathbb{Q})$ (so in particular $r+r'\leq n$).

 Then ${\rm pr}_3(T\cap {\rm pr}_{12}^{-1}(I_r\times I_{r'}))\subset X$ contains  an algebraically coisotropic $ I_{r+r'}$ subvariety of codimension $n-r-r'$.

\end{theo}
As usual, ${\rm pr}_i$ and  ${\rm pr}_{ij}$  denote the projections from $X\times X\times X$ to its factors, or products of two factors.
\begin{proof} The variety $Y:=T\cap {\rm pr}_{12}^{-1}(I_r\times I_{r'})\subset X\times X\times X$ maps  to $B_r\times B_{r'}$ by the map
$\phi_{r+r'}:=(\phi_r,\phi_{r'})\circ {\rm pr}_{12\mid Y}$.
By the definition of a triangle variety and using (\ref{eqalgcoisio}), we get that
\begin{eqnarray}
\label{eqnewcoisot} {\rm pr}_3^*\sigma_{X\mid Y}=-\phi_{r+r'}^*({\rm pr}_1^*\sigma_{B_r}+{\rm pr}_2^*\sigma_{B_{r'}})\,\,{\rm in}
\,\,H^{2,0}(Y_{\rm reg}).
\end{eqnarray}
It thus follows that the   rank of $ {\rm pr}_3^*\sigma_X$ restricted to  $Y_{\rm reg}$   is nowhere greater than $2r+2r'$.
On the other hand, Condition (*) implies that ${\rm pr}_3(Y)$ has at least one component of dimension
$\geq n+r+r'$. This component thus must have dimension exactly $n+r+r'$ and is coisotropic.
This is the desired variety $I_{r+r'}$, and it is in fact
algebraically coisotropic, choosing a subvariety $Y'\subset Y$ mapping to $I_{r+r'}$   in a generically finite way and using the diagram
\begin{eqnarray}
  \xymatrix{
&Y'\ar[r]^{{\rm pr}_3}\ar[d]^{(\phi_r,\phi_{r'})\circ {\rm pr}_{12}}& I_{r+r'}\\
&B_r\times B_{r'}&},
\end{eqnarray}
in which ${\rm pr}_3^*\sigma_{X\mid Y'}=\phi_r^*\sigma_{B_r}+\phi_{r'}^*\sigma_{B_{r'}}$.
\end{proof}
As a consequence of Theorem \ref{proexistesurfdec} (or using  methods similar  as above), we conclude now that many explicitly constructed projective hyper-K\"ahler manifolds admit a surface decomposition :
\begin{theo}\label{theodec} The following hyper-K\"ahler manifolds  admit surface decompositions:

\begin{enumerate}
\item\label{1} The Fano variety of lines $X=F_1(Y)$ of  a cubic fourfold $Y$ (see \cite{beaudo}).
\item\label{2} The Debarre-Voisin hyper-K\"ahler fourfold   (see \cite{debarrevoisin}).
\item\label{3} The double EPW sextics (see \cite{ogrady}).
\item\label{4} The LLSvS hyper-K\"ahler $8$-fold (see \cite{llss}).
\item\label{5}The LSV compactification of the intermediate Jacobian fibration associated with a cubic fourfold
(see \cite{lsv}).
\end{enumerate}
\end{theo}
\begin{proof}[Proof of cases   \ref{1} and \ref{2}] The case of  the Beauville-Donagi hyper-K\"ahler fourfold
$X=F_1(Y)$ is done as follows:
recall first that $F_1(Y) $ has an ample (singular) uniruled divisor
$D$ which can be constructed using  the rational self-map of degree $16$
$$\phi:X\dashrightarrow X$$
constructed in \cite{voisinFano}. This map associates to a general  point $[l]$ parameterizing
a line $l\subset Y$ the point  $[l']$ parameterizing
the  line $l'\subset Y$ such that there is a unique plane
$P\subset \mathbb{P}^5$ with $P\cap Y=2l+l'$. It satisfies the property that
\begin{eqnarray}
\label{eqpullbackphi} \phi^*\sigma_X=-2\sigma_X.
\end{eqnarray}
This map has indeterminacies
when the plane $P$ is not unique, and this happens along a surface $\Sigma$ which is studied in
\cite{amerik}. After blowing-up $\Sigma$, the map $\phi$ becomes a morphism
$\tilde{\phi}:\widetilde{F_1(Y)}\rightarrow F_1(Y)$ which is finite (see \cite{amerik}). The image
 of the exceptional divisor $E$ under $\tilde{\phi}$ is thus a uniruled divisor
 $E'$ in $F_1(Y)$  which has in fact   normalization isomorphic to $E$.
  We thus have a diagram

 \begin{eqnarray}\label{numerodiagpourEF}
  \xymatrix{
&E\ar[d]^{\tau_E}\ar^<{\,\,\,\,\,\,\,\,\,\,\tilde{\phi}_E}[r]&E'\subset X=F_1(Y)\\
&\Sigma&},
\end{eqnarray}
where $\tau_E$ is the restriction of the blowing-up morphism to $E$,
such that
$\tilde{\phi}_E^*\sigma_X=\tau_E^*\sigma_\Sigma$ for some
holomorphic $2$-form $\sigma_\Sigma$ on $\Sigma$.

On the other hand, we have the triangle variety
$T\subset X\times X\times X$ described in Section \ref{subsecvarofline}.
We thus have the ingredients needed to apply Theorem \ref{proexistesurfdec}, but we have to check
the condition (\ref{eqintnonnul}).
This    is easy because
${\rm pr}_{12}(T)\subset X\times X$ has codimension $2$, and
 the classes of the  fibers $F$ of $\tau_E:E'\rightarrow \Sigma$ must be proportional to $h^3$, where
 $h$ is the first Chern class of an ample line bundle on $X$,  because $\rho (X)=0$.
 Hence the intersection number
 $F\cdot {\rm pr}_{12}(T)$ is positive. We thus get a surface decomposition
 given by
  \begin{eqnarray}\label{eqdecompF1}\Gamma={\rm pr}_{12}^{-1}(E\times E )\cap T\subset X^3,\\
  \nonumber
  \psi=(\tau_E,\tau_E)\circ {\rm pr}_{12}:\Gamma\rightarrow \Sigma\times \Sigma,\,\,\phi={\rm pr}_{3\mid \Gamma}:\Gamma\rightarrow X.
\end{eqnarray}
The proof in the case \ref{2} works similarly. We use on the one hand the
triangle variety $T$ constructed by Bazhov (see \cite{bazhov} or  Section
\ref{secbazhov}), and on the other hand the existence of a uniruled divisor $\tau:D\rightarrow \Sigma,\,D\rightarrow X$
that we can exhibit either by
looking at the indeterminacies of Bazhov's construction, or by applying
\cite{charlespacienza}, using the fact that the Debarre-Voisin fourfold has the deformation type of $K3^{[2]}$. As the very general Debarre-Voisin varieties have Picard number $1$, the fibers
of a uniruled divisors have as before a class proportional to  $h^3$, where $h$ is an ample divisor class, hence
the variety
$${\rm pr}_{12}^{-1}(D\times D)\subset T\subset X\times X\times X$$
dominates $\Sigma\times \Sigma$ by $(\tau,\tau)\circ {\rm pr}_{12}$, hence $X$ by the projection ${\rm pr}_3$. The rest of the argument is identical.
\end{proof}
\begin{proof}[Proof of cases \ref{4} and \ref{5}]
The LLSvS manifold $F'_3(Y)$ is a hyper-K\"ahler $8$-fold contructed in
\cite{llss} as  a smooth hyper-K\"ahler model of the
basis of the rationally connected fibration of
the Hilbert scheme $F_3(Y)$ of degree $3$ rational curves in a smooth cubic fourfold
 $X$ not containing a plane.
 In the paper \cite{voicoiso}, we constructed a dominating rational map
 $$\mu: F_1(Y)\times F_1(Y)\dashrightarrow F'_3(Y).$$
 Two non intersecting lines $l,\,l'$ in $Y$ generate
 a $\mathbb{P}^3$ which intersects $Y$ along a cubic surface $S$.
 The $\mathbb{P}^2$ of degree $3$ rational curves
 corresponding to $\mu(l,l')$ is the linear system
 $|h+l-l'|$ on $S$, where $h=\mathcal{O}_S(1)$.
 It follows from this formula, Mumford's theorem and the fact that the
 holomorphic $2$-forms on the considered varieties come from a $(3,1)$-class on $Y$ by the corresponding  incidence
 correspondences,  that
 \begin{eqnarray}\label{eqmoinsmu*}\mu^*\sigma_{F'_3(Y)}={\rm pr}_1^*\sigma_{F_1(Y)}-{\rm pr}_2^*\sigma_{F_1(Y)}.
 \end{eqnarray}

 Together with Case \ref{1}, this immediately gives us a surface decomposition for
 $F'_3(Y)$. Indeed, we have the surface decomposition
 $\phi:\Gamma\rightarrow F_1(Y),\,\psi:\Gamma\rightarrow \Sigma\times \Sigma$ for $F_1(Y)$
 of (\ref{eqdecompF1}). The maps  satisfy
 \begin{eqnarray}\label{eqformetrie} \phi^*\sigma_{F_1(Y)}=\psi^*({\rm pr}_1^*\sigma_\Sigma+{\rm pr}_2^*\sigma_\Sigma).
 \end{eqnarray}

 Taking products, we get
 $$\phi':\Gamma\times \Gamma \rightarrow F_1(Y)\times F_1(Y),\,\,
 \psi':\Gamma\times \Gamma\rightarrow\Sigma\times \Sigma\times \Sigma\times\Sigma.$$
 Composing the first map with $\mu$ and desingularizing, we get
 \begin{eqnarray}
 \label{eqdefpoursurfF3}
 \phi'':\widetilde{\Gamma\times \Gamma} \rightarrow F_3(Y),\,\,
 \psi'':\widetilde{\Gamma\times \Gamma}\rightarrow\Sigma\times \Sigma\times \Sigma \times \Sigma.
 \end{eqnarray}
 By (\ref{eqformetrie}) and (\ref{eqmoinsmu*}), the morphisms in (\ref{eqdefpoursurfF3}) satisfy
 \begin{eqnarray}
 \label{eqdefpoursurfF3avecform} {\phi''}^*\sigma_{F_3(Y)}={\psi''}^*({\rm pr}_1^*\sigma_\Sigma+{\rm pr}_2^*\sigma_\Sigma-{\rm pr}_3^*\sigma_\Sigma-{\rm pr}_4^*\sigma_\Sigma),
 \end{eqnarray}
 which gives the desired decomposition in case \ref{4}.

 We now turn to the LSV hyper-K\"ahler fourfold $J(Y)$. It is a $10$-dimensional hyper-K\"ahler manifold associated to a general cubic fourfold $Y$. As it has a Lagrangian fibration, we will be able to use the triangle  variety described in Section \ref{secjacfib}.
 Another ingredient we will use is the following:
 \begin{lemm}\label{leFdansLSV} There exists a codimension $3$ algebraically coisotropic subvariety
 of $J(Y)$ which is birational to a $\mathbb{P}^{3}$-bundle over $F_1(Y)$.
 \end{lemm}
 \begin{proof} For each line $l\subset Y\subset  \mathbb{P}^5$, there is a $\mathbb{P}^3\subset(\mathbb{P}^5)^*$ of hyperplane sections of $Y$ containing $l$.
 This determines a $\mathbb{P}^3$-bundle $P\rightarrow F_1(Y)$.
 Each of these hyperplanes $H$ determines a hyperplane section $Y_H$ of
 $Y$. Then $l\subset Y_H$ and the $1$-cycle
 $3l-h^2$ is homologous to $0$ on $Y_H$, at least assuming
 $Y_H$ smooth, which allows to define a point
 $\Psi_{Y_H}(3l-h^2)\in J(Y_H)$ using the Abel-Jacobi map $\Psi_{Y_H}$ of $Y_H$.
 As $J(Y)$ is fibered over $(\mathbb{P}^5)^*$ into intermediate Jacobians
 $J(Y_H)$, we thus
 constructed the desired  rational map $P\dashrightarrow J(Y)$. It is not hard to see that
 this map is birational onto its image $P'$ which has thus dimension $7$. As the only holomorphic $2$-forms on $P$ are those coming from $F_1(Y)$, we conclude that $P'\subset J(Y)$ is
 algebraically coisotropic.
 \end{proof}
\begin{coro}\label{coropourIJ} There exists an algebraically coisotropic subvariety $Z$ of
$J(Y)$ which has codimension $4$ (dimension $6$). This variety dominates $(\mathbb{P}^5)^*$.
\end{coro}
\begin{proof}  We use for this the existence of the uniruled divisor $E'\subset F_1(Y)$ appearing in (\ref{numerodiagpourEF}).
Let now $P_{E'}$ be the inverse image of $E'$ in $P$ and let $Z$ be its image in $J(Y)$.
We have to prove that $Z$ dominates $(\mathbb{P}^5)^*$. This is saying that
any hyperplane section $Y_H$ of $Y$ contains a line residual to a special line
of $Y$, which is implied by the fact  that no hyperplane section of $Y$ contains
a $2$-dimensional family of lines residual to a special line
of $Y$. The last statement is proved in \cite{lsv}.
\end{proof}
We now construct a surface decomposition for $J(Y)$:
Let $Z$ be as in Corollary \ref{coropourIJ}, so $Z$ is fibered into curves over $B$.
We use the sum map  on the fibers of the fibration
$J(Y)\rightarrow B:=(\mathbb{P}^5)^*$.
We then get
a morphism:
$$\mu_{Z,5}:Z\times_B\ldots\times_BZ\dashrightarrow J(Y),$$
$$(a_1,\ldots,a_5)\mapsto \sum_ia_i.$$

We first observe that $\mu_5$ is dominant. This is because the fibers
of $J(Y)\rightarrow B$ are generically irreducible abelian varieties
and the fibers of $Z\rightarrow B$ are curves $Z_H$  which must generate
$J(Y_H)$, which is $5$-dimensional.
Finally, it remains to prove that the construction above provides
a surface decomposition for $J(Y)$. First of all, by Proposition \ref{triangalelagrangien},
for the relative sum map
$$\mu_5:J(Y)\times_B\ldots\times_BJ(Y)\dashrightarrow J(Y),$$
$$(a_1,\ldots,a_5)\mapsto \sum_ia_i,$$
one has
\begin{eqnarray}\label{equtile}\mu_5^*(\sigma_{J(Y)})=\sum_i{\rm pr}_i^*\sigma_{J(Y)}.
\end{eqnarray}
Next we use the rational map $f: Z\rightarrow \Sigma$ which is the composition of
$Z\rightarrow E'\subset F_1(Y)$, with $E'$ birational to $E$, and  ${\tau_E}:  E{\rightarrow} \Sigma$. We clearly have
\begin{eqnarray}\label{equtile2}f^*\sigma_\Sigma=\sigma_{J\mid Z}.
\end{eqnarray} Next $f$ induces
a morphism
$$f_5:Z\times_B\ldots\times_BZ\rightarrow \Sigma^5$$
and combining (\ref{equtile}) and (\ref{equtile2}), one concludes that
the diagram
\begin{eqnarray}\label{numerodiag}
  \xymatrix{
&Z\times_B\ldots\times_BZ\ar[d]^{f_5}\ar[r]^{\,\,\,\,\,\,\,\,\,\,\mu_5}&J(Y)\\
&\Sigma^5&}
\end{eqnarray}
provides a surface decomposition of $J(Y)$.
 \end{proof}

Coll\`{e}ge de France, 3 rue d'Ulm, 75005 Paris FRANCE

claire.voisin@imj-prg.fr
    \end{document}